\documentclass[11pt]{amsart} 

\usepackage[utf8]{inputenc} 

\usepackage[centering]{geometry} 
\usepackage{graphicx} 

\usepackage{booktabs} 
\usepackage{array} 
\usepackage{paralist} 
\usepackage{verbatim} 
\usepackage{subfig} 

\usepackage{braket}
\usepackage{dsfont}
\usepackage{faktor}
\usepackage{rotating}
\usepackage[arrow, matrix, curve]{xy}
\usepackage{mathtools}
\usepackage{hhline}
\usepackage{colortbl}
\usepackage{bbm}
\usepackage{stmaryrd}
\usepackage{paralist}
\usepackage{setspace}
\usepackage{hyperref}

\usepackage{tikz}
\usetikzlibrary{arrows,patterns,decorations.pathmorphing}
\tikzset{>=latex}
\usepackage{tkz-euclide}
\usetikzlibrary{shapes.misc}
\tikzset{cross/.style={cross out, draw=black, fill=none, minimum size=2*(#1-\pgflinewidth), inner sep=0pt, outer sep=0pt}, cross/.default={2pt}}

\usepackage{amsmath}
\usepackage{amssymb}
\usepackage{amsthm}

\DeclareMathOperator{\Hom}{Hom}
\DeclareMathOperator{\Aut}{Aut}

\newtheorem{prop}{Proposition}[section]
\newtheorem{thm}[prop]{Theorem}
\newtheorem{cor}[prop]{Corollary}
\newtheorem{lem}[prop]{Lemma}

\newtheorem{thmintro}{Theorem}
\newtheorem{corintro}{Corollary}

\theoremstyle{definition}
\newtheorem{dfn}[prop]{Definition}

\newtheorem{ex}[prop]{Example}

\newtheorem{notation}[prop]{Notation}

\newtheorem*{dfnintro}{Definition}

\theoremstyle{remark}
\newtheorem*{remintro}{Remark}
\newtheorem{rem}[prop]{Remark}

\title[Scattering diagrams: polynomiality and the dense region]{Scattering diagrams: \\ polynomiality and the dense region}

\author{Tim Gr\"afnitz}
\address{Leibniz-Universit\"at Hannover \\ Institut f\"ur Algebraische Geometrie \\ Welfengarten 1, 30167 Hannover \\ Germany} 
\email{graefnitz@math.uni-hannover.de}

\author{Patrick Luo} 
\address{Patrick Luo, University of Cambridge, DPMMS, Wilberforce Road, CB3 0WB, UK}
\email{pl485@cam.ac.uk}

\begin{document}
\maketitle

\begin{abstract}
We use deformations and mutations of scattering diagrams to show that the coefficients of a scattering diagram with initial functions $f_1=(1+tx)^\mu$ and $f_2=(1+ty)^\nu$ are polynomial in $\mu,\nu$ and non-trivial in a certain dense region. We discuss consequences for Gromov-Witten invariants and quiver representations.
\end{abstract}

\setcounter{tocdepth}{1}
\tableofcontents
\thispagestyle{empty}

\section*{Introduction}

Scattering diagrams (introduced in \cite{KS1}) are a method to combinatorially encode families of automorphisms of an algebraic torus (or, more generally, elements of the Lie group associated to a pro-nilpotent Lie algebra). They are related to various subjects such as curve counting \cite{KS2}\cite{GPS}\cite{Bou1}\cite{AG}\cite{BBG}\cite{Gra}, quiver representations \cite{Rei2}\cite{GP}, stability conditions \cite{Bri}\cite{Bou2}, cluster algebras \cite{GHKK} and mirror symmetry \cite{GSrec}\cite{GHK}\cite{GScan}. In this paper we will partly discuss the first two, but we will mainly be interested in scattering diagrams in their own right. 

A scattering diagram (in dimension $2$) is a collection of rays $\mathfrak{d}\subset\mathbb{R}^2$ with attached functions $f_{\mathfrak{d}}\in\mathbb{C}[x^{\pm1},y^{\pm1}]\llbracket t\rrbracket$. It is completely described by the coefficients $c_{a,b}$ of its functions. We use the factorized representation 
\[ f_{\mathfrak{d}}=\prod_{k>0}(1+tx^ay^b)^{c_{ka,kb}}. \]
Using a change of lattice trick (as in \cite{GHKK}) it is enough to consider ``standard'' scattering diagrams $\mathfrak{D}^{\mu,\nu}$ which contain rays with functions $f_1=(1+tx)^\mu$ and $f_2=(1+ty)^\nu$. We write the coefficients of such a diagram as $c_{a,b}^{\mu,\nu}$ and will be mainly concerned with these numbers. We use deformations of scattering diagrams (as described in \cite{GPS}\cite{GHKK}) and mutations (see e.g. \cite{GP}) to make statements about the coefficients $c_{a,b}^{\mu,\nu}$.

Our main theorems are the following.

\begin{dfnintro}
Define $(\mu,\nu)$-mutations $\textbf T_1^{\mu,\nu}(a,b)=(\mu b-a,b)$ and $\textbf T_2^{\mu,\nu}(a,b)=(a,\nu a-b)$.
\end{dfnintro}

\begin{dfnintro}
We say $(a,b)\in\mathbb{Z}_{>0}^2$ is in the dense region $\Phi^{\mu,\nu}$ if
\[\frac{\mu \nu - \sqrt{\mu \nu (\mu \nu - 4)}}{2 \mu} < \frac{b}{a} < \frac{\mu \nu + \sqrt{\mu \nu (\mu \nu - 4)}}{2 \mu}. \]
\end{dfnintro}

\begin{thmintro}
\label{thm:main}
\begin{enumerate}[(a)]
\item If $(a,b)\in\mathbb{Z}_{>0}^2$ is in the dense region, then $c_{a,b}^{\mu,\nu}\neq 0$.
\item Otherwise, $c_{a,b}^{\mu,\nu}\neq 0$ if and only if $(a,b)$ is obtained from $(1,0)$ or $(0,1)$ via a sequence of $(\mu,\nu)$-mutations. In particular, $(a,b)$ must be primitive in this case.
\end{enumerate}
\end{thmintro}

\begin{remintro}
The above statement has been proved in the case $\mu=\nu$ in \cite[\S4.7]{GP}, using an existence statement for quiver representations provided by Reineke. Reineke told us that there should be a similar argument in the case $\mu\neq \nu$ using bipartite quivers \cite{RW}, but this has not been worked out in detail. Our proof is purely combinatorial in terms of scattering diagrams and does not make use of quiver representations.
\end{remintro}

The above theorem implies the existence of certain quiver representations: 

\begin{dfnintro}
Let $M_{(a,b)}^{(1,0)-st}(Q_{\mu,\nu})$ be the moduli space of (isomorphism classes of) $(1,0)$-stable representations of the complete bipartite quiver with $\mu+\nu$ vertices (see \S\ref{S:quiv}).
\end{dfnintro}

\begin{corintro}
If $(a,b)$ is in the dense region (i.e. satisfies the inequality above), then $M_{(a,b)}^{(1,0)-st}(Q_{\mu,\nu})$ has positive Euler characteristic and in particular is non-empty.
\end{corintro}

\begin{thmintro}
\label{polythm}
The coefficients \(c_{a,b}^{\mu, \nu}\) are polynomial in \(\mu, \nu\) of degrees \(a, b\) respectively. In the binomial expansion
\[ c_{a,b}^{\mu,\nu} = \sum_{k=1}^\mu\sum_{l=1}^\nu\lambda_{k,l}\binom{\mu}{k}\binom{\nu}{l} \]
we have $\lambda_{k,l}=0$ whenever $c_{a,b}^{k,l}=0$.
\end{thmintro}

The above theorem implies the vanishing of certain Gromov-Witten invariants:

\begin{dfnintro}
Let $\vec{p}_1,\vec{p}_2$ be tuples of integers of lengths $\ell_1,\ell_2$ and summing to $a$ and $b$, respectively. Let $N_{\vec{p}_1,\vec{p}_2}$ be the relative Gromov-Witten invariant on the blow up of $\mathbb{P}(1,a,b)$ in $\ell_1$ points on $D_a$ and $\ell_2$ points on $D_b$, counting rational curves of class $\beta=\gcd(a,b)H-\sum_{j=1}^{\ell_1}p_{1j}E_{1j}-\sum_{j=1}^{\ell_2}p_{2j}E_{2j}$ (see \S\ref{S:GW}).
\end{dfnintro}

\begin{corintro}
If $\gcd(a,b)=1$ and $c_{a,b}^{\mu,\nu}=0$, then $N_{\vec{p}_1,\vec{p}_2}=0$ for all $\vec{p}_1,\vec{p}_2$ of lengths $\mu,\nu$ and summing to $a,b$, respectively.
\end{corintro}

\subsection*{Acknowledgements}

This paper grew out of a Summer Research in Mathematics project carried out at the University of Cambridge. This project was financially supported by Mark Gross' ERC Advanced Grant Mirror Symmetry in Algebraic Geometry (MSAG). We thank Dhruv Ranganathan for bringing us together and Markus Reineke for clarifying some questions about quiver representations.

\section{Preliminaries}

\subsection{Scattering diagrams}
We provide a definition for scattering diagrams, based on \cite{GPS}. See \cite{GPS}, \cite{GHKK} for more general definitions.

Let \(M \cong \mathbb Z^2\) be a lattice with basis \(e_1 = (1, 0), e_2 = (0, 1)\), and let \(N := \Hom_{\mathbb Z} (M, \mathbb Z)\). For \(m \in M\), let \(z^m \in \mathbb C[M]\) denote the corresponding element in the group ring. With \(x = z^{e_1}, y = z^{e_2}\), \(\mathbb C[M] = \mathbb C[x^{\pm 1}, y^{\pm 1}]\) is the ring of Laurent polynomials in \(x\) and \(y\).

Let \(R\) be an Artin local $\mathbb{C}$-algebra with maximal ideal \(\mathfrak m_R\), and \[\mathbb C[M] \widehat{\otimes}_{\mathbb C} R = \lim_{\longleftarrow} \mathbb C[M] \otimes_{\mathbb C} R/\mathfrak m_R^k .\]
We will take \(M = N = \mathbb Z^2\) and \(R = \mathbb C \llbracket t \rrbracket\) unless otherwise specified. In this case \(\mathbb C[M] \widehat{\otimes}_{\mathbb C} R = \mathbb C[x^{\pm 1}, y^{\pm 1}] \llbracket t \rrbracket\).

\begin{dfn}
A \emph{ray} or \emph{line} is a pair \( \mathfrak{d} = (\underline{\mathfrak{d}}, f_{\mathfrak d})\) where \(\underline{\mathfrak{d}} = b_{\mathfrak{d}} + \mathbb R_{\geq 0} m_{\mathfrak{d}}\) if it's a ray or \(\underline{\mathfrak{d}} = b_{\mathfrak{d}} + \mathbb R m_{\mathfrak{d}}\) if it's a line, \[f_{\mathfrak d} \in \mathbb C[z^{m_{\mathfrak{d}}}] \widehat{\otimes}_{\mathbb C} R \setminus \{1\} \subseteq \mathbb C[M] \widehat{\otimes}_{\mathbb C} R ,\] and \[f \equiv 1 \pmod{z^{m_{\mathfrak{d}}} \mathfrak m_R} .\] We can consider a line to be two rays in opposite directions originating from a common point, with the same function $f_{\mathfrak{d}}$.

A \emph{scattering diagram} \(\mathfrak D\) is a collection of rays and lines, such that for every \(k > 0\), there are finitely many rays and lines \((\underline{\mathfrak{d}}, f_{\mathfrak d})\) with \(f_{\mathfrak d} \not\equiv 1 \pmod{\mathfrak m_R^k}\).
\end{dfn}

\begin{dfn}
For a ray \(\mathfrak{d}\) and a curve \(\gamma\) in \(M_{\mathbb R}\) intersecting \(\mathfrak d\) transversally at \(p\), let \(n_{\mathfrak d} \in N\) annihilate \(m_{\mathfrak d}\) and evaluate positively on \(\gamma'(p)\). Define \(\theta_{\mathfrak d} = \theta_{\gamma, p, \mathfrak d} \in \Aut_{\mathbb C \llbracket t \rrbracket} (\mathbb C[M] \widehat{\otimes}_{\mathbb C} R)\) by \[\theta_{\mathfrak d} : z^m \mapsto z^m f_{\mathfrak d}^{\langle m, n_{\mathfrak d} \rangle} .\]
\end{dfn}

\begin{dfn}
A \emph{singularity} of a scattering diagram \(\mathfrak D\) is either a base point of a ray or an intersection between two rays or lines that consists of a single point.

Let \(\gamma : [0, 1] \to M_{\mathbb R}\) be a smooth curve which does not pass through any singularities and whose endpoints are not in any ray or line in the diagram. If all intersections of \(\gamma\) with rays or lines are transverse, we define the \emph{\(\gamma\)-ordered product} \(\theta_{\gamma, \mathfrak D} \in \Aut_R (\mathbb C[M] \widehat{\otimes}_{\mathbb C} R)\) in the following way. For each \(k\), as there are finitely many rays or lines with functions \(f_{\mathfrak d} \not\equiv 1 \pmod{\mathfrak m_R^k}\), let \(0 < p_1 \leq p_2 \leq \dots \leq p_s < 1\) be so that at each \(p_i\), \(\gamma(p_i) \in \mathfrak d_i\) for some ray or line \((\mathfrak d_i, f_{\mathfrak d_i})\), and when \(p_i = p_j\) for \(i \neq j\), \(\mathfrak d_i \neq \mathfrak d_j\) are different rays of the diagram, though may have the same set of points and function, where \(s\) is chosen to be as large as possible. Then let \(\theta_i = \theta_{\gamma, p_i, \mathfrak d_i}\) and \[\theta_{\gamma, \mathfrak D}^k = \theta_s \circ \dots \circ \theta_2 \circ \theta_1 .\] Then we define $\theta_{\gamma, \mathfrak D}$ as the formal limit \[\theta_{\gamma, \mathfrak D} = \lim_{k \to \infty} (\theta_{\gamma, \mathfrak D}^k) .\]

We say a diagram \(\mathfrak D\) is \emph{consistent} if \(\theta_{\gamma, \mathfrak D}\) is the identity map for every closed curve \(\gamma\) (for which \(\theta_{\gamma, \mathfrak D}\) is defined). Two diagrams $\mathfrak{D}$ and $\mathfrak{D}'$ are \emph{equivalent} if $\theta_{\gamma, \mathfrak D} = \theta_{\gamma, \mathfrak D'}$ for every closed curve $\gamma$.
\end{dfn}

\begin{prop}[\cite{KS1}, {\cite[Theorem 1.4]{GPS}}]
For a scattering diagram \(\mathfrak D\), there exists a consistent scattering diagram \(\mathfrak D_\infty \supseteq \mathfrak D\) such that \(\mathfrak D_\infty \setminus \mathfrak D\) consists only of rays.
\end{prop}

\begin{rem}
\label{rem:unique}
The consistent diagram $\mathfrak{D}_\infty$ obtained from $\mathfrak{D}$ is unique (up to equivalence), if we require that it has no two rays $\mathfrak{d},\mathfrak{d}'$ with the same support $\underline{\mathfrak{d}}=\underline{\mathfrak{d}}'$.
\end{rem}

\begin{dfn}
If a consistent diagram $\mathfrak{D}=\mathfrak{D}_\infty$ has only one singularity, then (by Remark \ref{rem:unique}, up to equivalence) there is a unique ray in each direction $m\in\mathbb{Z}^2$. We write the function of this ray as $f_m^{\mathfrak{D}}$.
\end{dfn}

\begin{dfn}
\label{def:standard}
The \emph{standard scattering diagram} \(\mathfrak D^{\mu, \nu} = \mathfrak D_\infty^{\mu, \nu}\) is the diagram obtained by performing scattering on the initial diagram \[\mathfrak D_0^{\mu, \nu} = \{(\mathbb R (1, 0), (1 + t x)^\mu), (\mathbb R (0, 1), (1 + t y)^\nu)\}. \]
\end{dfn}

The scattering only produces rays in the first quadrant, i.e. with $m_{\mathfrak{d}}=(a,b)\in\mathbb{Z}_{>0}^2$. Consider an equivalent diagram to a standard scattering diagram such that there is a unique ray in each direction (see Remark \ref{rem:unique}). We can express the function $f_{\mathfrak{d}}$ of the ray $\mathfrak{d}$ in direction $(a,b)\in\mathbb{Z}_{>0}^2$ as 
\[f_{(a, b)}^{\mu, \nu} := f_{(a,b)}^{\mathfrak{D}^{\mu,\nu}} = \prod_{k=1}^\infty (1 + t^{ka + kb} x^{ka} y^{kb})^{c_{ka, kb}^{\mu, \nu}}. \]
\begin{dfn}
\label{def:c}
The \emph{coefficients} for \(\mathfrak D^{\mu, \nu}\) are these \(c_{a, b}^{\mu, \nu}\).
\end{dfn}

\begin{prop}[{\cite[Proposition C.13]{GHKK}}]
\label{prop:pos}
The coefficients of a standard scattering diagram are all positive:
\[c_{a, b}^{\mu, \nu} \geq 0.\]
\end{prop}

\begin{notation}
When talking about standard scattering diagrams $D^{\mu,\nu}$ we will often omit the factors of \(t\) for better readability. We can do this, because all monomials are of the form $t^{a+b}x^ay^b$, i.e. the $t$-power can be read off from the $x$- and $y$-powers.
\end{notation}

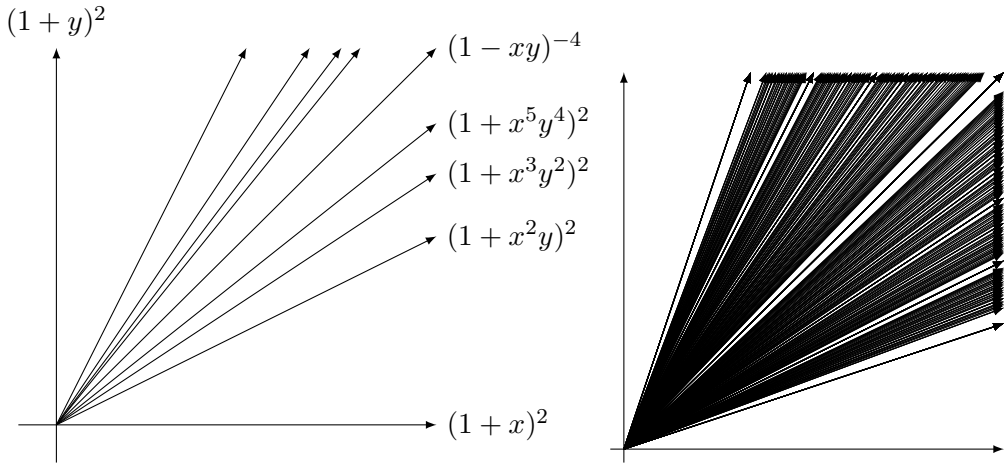
\begin{figure}[h!]
\centering
\begin{tikzpicture}[scale=1]
\draw[->] (0.0, -1/2) -- (0.0, 5.0) node[above]{$(1+y)^2$};
\draw[->] (-1/2, 0.0) -- (5.0, 0.0) node[right]{$(1+x)^2$};
\draw[->] (0.0, 0.0) -- (5.0, 5.0) node[right]{$(1-xy)^{-4}$};
\draw[->] (0.0, 0.0) -- (2.5, 5.0);
\draw[->] (0.0, 0.0) -- (5.0, 2.5) node[right]{$(1+x^2y)^2$};
\draw[->] (0.0, 0.0) -- (3.333, 5.0);
\draw[->] (0.0, 0.0) -- (5.0, 3.333) node[right]{$(1+x^3y^2)^2$};
\draw[->] (0.0, 0.0) -- (3.75, 5.0);
\draw[->] (0.0, 0.0) -- (4.0, 5.0);
\draw[->] (0.0, 0.0) -- (5.0, 4.0) node[right]{$(1+x^5y^4)^2$};
\end{tikzpicture}
\begin{tikzpicture}[scale=5/28]
\draw[->] (0.0, -1.0) -- (0.0, 28.0);
\draw[->] (-1.0, 0.0) -- (28.0, 0.0);
\draw[->] (0.0, 0.0) -- (28.0, 28.0);
\draw[->] (0.0, 0.0) -- (28.0, 28.0);
\draw[->] (0.0, 0.0) -- (28.0, 28.0);
\draw[->] (0.0, 0.0) -- (28.0, 28.0);
\draw[->] (0.0, 0.0) -- (28.0, 28.0);
\draw[->] (0.0, 0.0) -- (28.0, 28.0);
\draw[->] (0.0, 0.0) -- (28.0, 28.0);
\draw[->] (0.0, 0.0) -- (28.0, 28.0);
\draw[->] (0.0, 0.0) -- (28.0, 28.0);
\draw[->] (0.0, 0.0) -- (28.0, 28.0);
\draw[->] (0.0, 0.0) -- (28.0, 28.0);
\draw[->] (0.0, 0.0) -- (28.0, 28.0);
\draw[->] (0.0, 0.0) -- (28.0, 28.0);
\draw[->] (0.0, 0.0) -- (28.0, 28.0);
\draw[->] (0.0, 0.0) -- (28.0, 28.0);
\draw[->] (0.0, 0.0) -- (28.0, 28.0);
\draw[->] (0.0, 0.0) -- (28.0, 28.0);
\draw[->] (0.0, 0.0) -- (28.0, 28.0);
\draw[->] (0.0, 0.0) -- (28.0, 28.0);
\draw[->] (0.0, 0.0) -- (28.0, 28.0);
\draw[->] (0.0, 0.0) -- (14.0, 28.0);
\draw[->] (0.0, 0.0) -- (14.0, 28.0);
\draw[->] (0.0, 0.0) -- (14.0, 28.0);
\draw[->] (0.0, 0.0) -- (14.0, 28.0);
\draw[->] (0.0, 0.0) -- (14.0, 28.0);
\draw[->] (0.0, 0.0) -- (14.0, 28.0);
\draw[->] (0.0, 0.0) -- (14.0, 28.0);
\draw[->] (0.0, 0.0) -- (14.0, 28.0);
\draw[->] (0.0, 0.0) -- (14.0, 28.0);
\draw[->] (0.0, 0.0) -- (14.0, 28.0);
\draw[->] (0.0, 0.0) -- (14.0, 28.0);
\draw[->] (0.0, 0.0) -- (14.0, 28.0);
\draw[->] (0.0, 0.0) -- (14.0, 28.0);
\draw[->] (0.0, 0.0) -- (9.333, 28.0);
\draw[->] (0.0, 0.0) -- (9.333, 28.0);
\draw[->] (0.0, 0.0) -- (9.333, 28.0);
\draw[->] (0.0, 0.0) -- (9.333, 28.0);
\draw[->] (0.0, 0.0) -- (9.333, 28.0);
\draw[->] (0.0, 0.0) -- (9.333, 28.0);
\draw[->] (0.0, 0.0) -- (9.333, 28.0);
\draw[->] (0.0, 0.0) -- (9.333, 28.0);
\draw[->] (0.0, 0.0) -- (9.333, 28.0);
\draw[->] (0.0, 0.0) -- (9.333, 28.0);
\draw[->] (0.0, 0.0) -- (28.0, 14.0);
\draw[->] (0.0, 0.0) -- (28.0, 14.0);
\draw[->] (0.0, 0.0) -- (28.0, 14.0);
\draw[->] (0.0, 0.0) -- (28.0, 14.0);
\draw[->] (0.0, 0.0) -- (28.0, 14.0);
\draw[->] (0.0, 0.0) -- (28.0, 14.0);
\draw[->] (0.0, 0.0) -- (28.0, 14.0);
\draw[->] (0.0, 0.0) -- (28.0, 14.0);
\draw[->] (0.0, 0.0) -- (28.0, 14.0);
\draw[->] (0.0, 0.0) -- (28.0, 14.0);
\draw[->] (0.0, 0.0) -- (28.0, 14.0);
\draw[->] (0.0, 0.0) -- (28.0, 14.0);
\draw[->] (0.0, 0.0) -- (28.0, 14.0);
\draw[->] (0.0, 0.0) -- (18.667, 28.0);
\draw[->] (0.0, 0.0) -- (18.667, 28.0);
\draw[->] (0.0, 0.0) -- (18.667, 28.0);
\draw[->] (0.0, 0.0) -- (18.667, 28.0);
\draw[->] (0.0, 0.0) -- (18.667, 28.0);
\draw[->] (0.0, 0.0) -- (18.667, 28.0);
\draw[->] (0.0, 0.0) -- (18.667, 28.0);
\draw[->] (0.0, 0.0) -- (18.667, 28.0);
\draw[->] (0.0, 0.0) -- (11.2, 28.0);
\draw[->] (0.0, 0.0) -- (11.2, 28.0);
\draw[->] (0.0, 0.0) -- (11.2, 28.0);
\draw[->] (0.0, 0.0) -- (11.2, 28.0);
\draw[->] (0.0, 0.0) -- (11.2, 28.0);
\draw[->] (0.0, 0.0) -- (28.0, 9.333);
\draw[->] (0.0, 0.0) -- (28.0, 9.333);
\draw[->] (0.0, 0.0) -- (28.0, 9.333);
\draw[->] (0.0, 0.0) -- (28.0, 9.333);
\draw[->] (0.0, 0.0) -- (28.0, 9.333);
\draw[->] (0.0, 0.0) -- (28.0, 9.333);
\draw[->] (0.0, 0.0) -- (28.0, 9.333);
\draw[->] (0.0, 0.0) -- (28.0, 9.333);
\draw[->] (0.0, 0.0) -- (28.0, 9.333);
\draw[->] (0.0, 0.0) -- (28.0, 9.333);
\draw[->] (0.0, 0.0) -- (28.0, 18.667);
\draw[->] (0.0, 0.0) -- (28.0, 18.667);
\draw[->] (0.0, 0.0) -- (28.0, 18.667);
\draw[->] (0.0, 0.0) -- (28.0, 18.667);
\draw[->] (0.0, 0.0) -- (28.0, 18.667);
\draw[->] (0.0, 0.0) -- (28.0, 18.667);
\draw[->] (0.0, 0.0) -- (28.0, 18.667);
\draw[->] (0.0, 0.0) -- (28.0, 18.667);
\draw[->] (0.0, 0.0) -- (21.0, 28.0);
\draw[->] (0.0, 0.0) -- (21.0, 28.0);
\draw[->] (0.0, 0.0) -- (21.0, 28.0);
\draw[->] (0.0, 0.0) -- (21.0, 28.0);
\draw[->] (0.0, 0.0) -- (21.0, 28.0);
\draw[->] (0.0, 0.0) -- (16.8, 28.0);
\draw[->] (0.0, 0.0) -- (16.8, 28.0);
\draw[->] (0.0, 0.0) -- (16.8, 28.0);
\draw[->] (0.0, 0.0) -- (16.8, 28.0);
\draw[->] (0.0, 0.0) -- (16.8, 28.0);
\draw[->] (0.0, 0.0) -- (12.0, 28.0);
\draw[->] (0.0, 0.0) -- (12.0, 28.0);
\draw[->] (0.0, 0.0) -- (12.0, 28.0);
\draw[->] (0.0, 0.0) -- (12.0, 28.0);
\draw[->] (0.0, 0.0) -- (10.5, 28.0);
\draw[->] (0.0, 0.0) -- (10.5, 28.0);
\draw[->] (0.0, 0.0) -- (10.5, 28.0);
\draw[->] (0.0, 0.0) -- (28.0, 21.0);
\draw[->] (0.0, 0.0) -- (28.0, 21.0);
\draw[->] (0.0, 0.0) -- (28.0, 21.0);
\draw[->] (0.0, 0.0) -- (28.0, 21.0);
\draw[->] (0.0, 0.0) -- (28.0, 21.0);
\draw[->] (0.0, 0.0) -- (22.4, 28.0);
\draw[->] (0.0, 0.0) -- (22.4, 28.0);
\draw[->] (0.0, 0.0) -- (22.4, 28.0);
\draw[->] (0.0, 0.0) -- (22.4, 28.0);
\draw[->] (0.0, 0.0) -- (16.0, 28.0);
\draw[->] (0.0, 0.0) -- (16.0, 28.0);
\draw[->] (0.0, 0.0) -- (16.0, 28.0);
\draw[->] (0.0, 0.0) -- (12.444, 28.0);
\draw[->] (0.0, 0.0) -- (12.444, 28.0);
\draw[->] (0.0, 0.0) -- (12.444, 28.0);
\draw[->] (0.0, 0.0) -- (28.0, 11.2);
\draw[->] (0.0, 0.0) -- (28.0, 11.2);
\draw[->] (0.0, 0.0) -- (28.0, 11.2);
\draw[->] (0.0, 0.0) -- (28.0, 11.2);
\draw[->] (0.0, 0.0) -- (28.0, 11.2);
\draw[->] (0.0, 0.0) -- (28.0, 16.8);
\draw[->] (0.0, 0.0) -- (28.0, 16.8);
\draw[->] (0.0, 0.0) -- (28.0, 16.8);
\draw[->] (0.0, 0.0) -- (28.0, 16.8);
\draw[->] (0.0, 0.0) -- (28.0, 16.8);
\draw[->] (0.0, 0.0) -- (28.0, 22.4);
\draw[->] (0.0, 0.0) -- (28.0, 22.4);
\draw[->] (0.0, 0.0) -- (28.0, 22.4);
\draw[->] (0.0, 0.0) -- (28.0, 22.4);
\draw[->] (0.0, 0.0) -- (23.333, 28.0);
\draw[->] (0.0, 0.0) -- (23.333, 28.0);
\draw[->] (0.0, 0.0) -- (23.333, 28.0);
\draw[->] (0.0, 0.0) -- (20.0, 28.0);
\draw[->] (0.0, 0.0) -- (20.0, 28.0);
\draw[->] (0.0, 0.0) -- (20.0, 28.0);
\draw[->] (0.0, 0.0) -- (17.5, 28.0);
\draw[->] (0.0, 0.0) -- (17.5, 28.0);
\draw[->] (0.0, 0.0) -- (17.5, 28.0);
\draw[->] (0.0, 0.0) -- (15.556, 28.0);
\draw[->] (0.0, 0.0) -- (15.556, 28.0);
\draw[->] (0.0, 0.0) -- (12.727, 28.0);
\draw[->] (0.0, 0.0) -- (12.727, 28.0);
\draw[->] (0.0, 0.0) -- (11.667, 28.0);
\draw[->] (0.0, 0.0) -- (11.667, 28.0);
\draw[->] (0.0, 0.0) -- (10.769, 28.0);
\draw[->] (0.0, 0.0) -- (10.769, 28.0);
\draw[->] (0.0, 0.0) -- (28.0, 23.333);
\draw[->] (0.0, 0.0) -- (28.0, 23.333);
\draw[->] (0.0, 0.0) -- (28.0, 23.333);
\draw[->] (0.0, 0.0) -- (24.0, 28.0);
\draw[->] (0.0, 0.0) -- (24.0, 28.0);
\draw[->] (0.0, 0.0) -- (24.0, 28.0);
\draw[->] (0.0, 0.0) -- (15.273, 28.0);
\draw[->] (0.0, 0.0) -- (15.273, 28.0);
\draw[->] (0.0, 0.0) -- (12.923, 28.0);
\draw[->] (0.0, 0.0) -- (12.923, 28.0);
\draw[->] (0.0, 0.0) -- (28.0, 12.0);
\draw[->] (0.0, 0.0) -- (28.0, 12.0);
\draw[->] (0.0, 0.0) -- (28.0, 12.0);
\draw[->] (0.0, 0.0) -- (28.0, 12.0);
\draw[->] (0.0, 0.0) -- (28.0, 16.0);
\draw[->] (0.0, 0.0) -- (28.0, 16.0);
\draw[->] (0.0, 0.0) -- (28.0, 16.0);
\draw[->] (0.0, 0.0) -- (28.0, 20.0);
\draw[->] (0.0, 0.0) -- (28.0, 20.0);
\draw[->] (0.0, 0.0) -- (28.0, 20.0);
\draw[->] (0.0, 0.0) -- (28.0, 24.0);
\draw[->] (0.0, 0.0) -- (28.0, 24.0);
\draw[->] (0.0, 0.0) -- (28.0, 24.0);
\draw[->] (0.0, 0.0) -- (24.5, 28.0);
\draw[->] (0.0, 0.0) -- (24.5, 28.0);
\draw[->] (0.0, 0.0) -- (21.778, 28.0);
\draw[->] (0.0, 0.0) -- (21.778, 28.0);
\draw[->] (0.0, 0.0) -- (19.6, 28.0);
\draw[->] (0.0, 0.0) -- (19.6, 28.0);
\draw[->] (0.0, 0.0) -- (17.818, 28.0);
\draw[->] (0.0, 0.0) -- (17.818, 28.0);
\draw[->] (0.0, 0.0) -- (16.333, 28.0);
\draw[->] (0.0, 0.0) -- (16.333, 28.0);
\draw[->] (0.0, 0.0) -- (15.077, 28.0);
\draw[->] (0.0, 0.0) -- (15.077, 28.0);
\draw[->] (0.0, 0.0) -- (13.067, 28.0);
\draw[->] (0.0, 0.0) -- (12.25, 28.0);
\draw[->] (0.0, 0.0) -- (11.529, 28.0);
\draw[->] (0.0, 0.0) -- (10.889, 28.0);
\draw[->] (0.0, 0.0) -- (28.0, 10.5);
\draw[->] (0.0, 0.0) -- (28.0, 10.5);
\draw[->] (0.0, 0.0) -- (28.0, 10.5);
\draw[->] (0.0, 0.0) -- (28.0, 17.5);
\draw[->] (0.0, 0.0) -- (28.0, 17.5);
\draw[->] (0.0, 0.0) -- (28.0, 17.5);
\draw[->] (0.0, 0.0) -- (28.0, 24.5);
\draw[->] (0.0, 0.0) -- (28.0, 24.5);
\draw[->] (0.0, 0.0) -- (24.889, 28.0);
\draw[->] (0.0, 0.0) -- (24.889, 28.0);
\draw[->] (0.0, 0.0) -- (20.364, 28.0);
\draw[->] (0.0, 0.0) -- (20.364, 28.0);
\draw[->] (0.0, 0.0) -- (17.231, 28.0);
\draw[->] (0.0, 0.0) -- (14.933, 28.0);
\draw[->] (0.0, 0.0) -- (13.176, 28.0);
\draw[->] (0.0, 0.0) -- (11.789, 28.0);
\draw[->] (0.0, 0.0) -- (10.667, 28.0);
\draw[->] (0.0, 0.0) -- (28.0, 12.444);
\draw[->] (0.0, 0.0) -- (28.0, 12.444);
\draw[->] (0.0, 0.0) -- (28.0, 12.444);
\draw[->] (0.0, 0.0) -- (28.0, 15.556);
\draw[->] (0.0, 0.0) -- (28.0, 15.556);
\draw[->] (0.0, 0.0) -- (28.0, 21.778);
\draw[->] (0.0, 0.0) -- (28.0, 21.778);
\draw[->] (0.0, 0.0) -- (28.0, 24.889);
\draw[->] (0.0, 0.0) -- (28.0, 24.889);
\draw[->] (0.0, 0.0) -- (25.2, 28.0);
\draw[->] (0.0, 0.0) -- (25.2, 28.0);
\draw[->] (0.0, 0.0) -- (22.909, 28.0);
\draw[->] (0.0, 0.0) -- (22.909, 28.0);
\draw[->] (0.0, 0.0) -- (19.385, 28.0);
\draw[->] (0.0, 0.0) -- (18.0, 28.0);
\draw[->] (0.0, 0.0) -- (15.75, 28.0);
\draw[->] (0.0, 0.0) -- (14.824, 28.0);
\draw[->] (0.0, 0.0) -- (13.263, 28.0);
\draw[->] (0.0, 0.0) -- (12.6, 28.0);
\draw[->] (0.0, 0.0) -- (11.455, 28.0);
\draw[->] (0.0, 0.0) -- (10.957, 28.0);
\draw[->] (0.0, 0.0) -- (28.0, 19.6);
\draw[->] (0.0, 0.0) -- (28.0, 19.6);
\draw[->] (0.0, 0.0) -- (28.0, 25.2);
\draw[->] (0.0, 0.0) -- (28.0, 25.2);
\draw[->] (0.0, 0.0) -- (25.455, 28.0);
\draw[->] (0.0, 0.0) -- (21.538, 28.0);
\draw[->] (0.0, 0.0) -- (16.471, 28.0);
\draw[->] (0.0, 0.0) -- (14.737, 28.0);
\draw[->] (0.0, 0.0) -- (13.333, 28.0);
\draw[->] (0.0, 0.0) -- (12.174, 28.0);
\draw[->] (0.0, 0.0) -- (28.0, 12.727);
\draw[->] (0.0, 0.0) -- (28.0, 12.727);
\draw[->] (0.0, 0.0) -- (28.0, 15.273);
\draw[->] (0.0, 0.0) -- (28.0, 15.273);
\draw[->] (0.0, 0.0) -- (28.0, 17.818);
\draw[->] (0.0, 0.0) -- (28.0, 17.818);
\draw[->] (0.0, 0.0) -- (28.0, 20.364);
\draw[->] (0.0, 0.0) -- (28.0, 20.364);
\draw[->] (0.0, 0.0) -- (28.0, 22.909);
\draw[->] (0.0, 0.0) -- (28.0, 22.909);
\draw[->] (0.0, 0.0) -- (28.0, 25.455);
\draw[->] (0.0, 0.0) -- (25.667, 28.0);
\draw[->] (0.0, 0.0) -- (23.692, 28.0);
\draw[->] (0.0, 0.0) -- (22.0, 28.0);
\draw[->] (0.0, 0.0) -- (20.533, 28.0);
\draw[->] (0.0, 0.0) -- (19.25, 28.0);
\draw[->] (0.0, 0.0) -- (18.118, 28.0);
\draw[->] (0.0, 0.0) -- (17.111, 28.0);
\draw[->] (0.0, 0.0) -- (16.211, 28.0);
\draw[->] (0.0, 0.0) -- (15.4, 28.0);
\draw[->] (0.0, 0.0) -- (14.667, 28.0);
\draw[->] (0.0, 0.0) -- (13.391, 28.0);
\draw[->] (0.0, 0.0) -- (12.833, 28.0);
\draw[->] (0.0, 0.0) -- (12.32, 28.0);
\draw[->] (0.0, 0.0) -- (11.846, 28.0);
\draw[->] (0.0, 0.0) -- (11.407, 28.0);
\draw[->] (0.0, 0.0) -- (11.0, 28.0);
\draw[->] (0.0, 0.0) -- (28.0, 11.667);
\draw[->] (0.0, 0.0) -- (28.0, 11.667);
\draw[->] (0.0, 0.0) -- (28.0, 16.333);
\draw[->] (0.0, 0.0) -- (28.0, 16.333);
\draw[->] (0.0, 0.0) -- (28.0, 25.667);
\draw[->] (0.0, 0.0) -- (25.846, 28.0);
\draw[->] (0.0, 0.0) -- (19.765, 28.0);
\draw[->] (0.0, 0.0) -- (17.684, 28.0);
\draw[->] (0.0, 0.0) -- (14.609, 28.0);
\draw[->] (0.0, 0.0) -- (13.44, 28.0);
\draw[->] (0.0, 0.0) -- (28.0, 10.769);
\draw[->] (0.0, 0.0) -- (28.0, 10.769);
\draw[->] (0.0, 0.0) -- (28.0, 12.923);
\draw[->] (0.0, 0.0) -- (28.0, 12.923);
\draw[->] (0.0, 0.0) -- (28.0, 15.077);
\draw[->] (0.0, 0.0) -- (28.0, 15.077);
\draw[->] (0.0, 0.0) -- (28.0, 17.231);
\draw[->] (0.0, 0.0) -- (28.0, 19.385);
\draw[->] (0.0, 0.0) -- (28.0, 21.538);
\draw[->] (0.0, 0.0) -- (28.0, 23.692);
\draw[->] (0.0, 0.0) -- (28.0, 25.846);
\draw[->] (0.0, 0.0) -- (26.0, 28.0);
\draw[->] (0.0, 0.0) -- (24.267, 28.0);
\draw[->] (0.0, 0.0) -- (22.75, 28.0);
\draw[->] (0.0, 0.0) -- (21.412, 28.0);
\draw[->] (0.0, 0.0) -- (20.222, 28.0);
\draw[->] (0.0, 0.0) -- (19.158, 28.0);
\draw[->] (0.0, 0.0) -- (18.2, 28.0);
\draw[->] (0.0, 0.0) -- (17.333, 28.0);
\draw[->] (0.0, 0.0) -- (16.545, 28.0);
\draw[->] (0.0, 0.0) -- (15.826, 28.0);
\draw[->] (0.0, 0.0) -- (15.167, 28.0);
\draw[->] (0.0, 0.0) -- (14.56, 28.0);
\draw[->] (0.0, 0.0) -- (13.481, 28.0);
\draw[->] (0.0, 0.0) -- (28.0, 18.0);
\draw[->] (0.0, 0.0) -- (28.0, 22.0);
\draw[->] (0.0, 0.0) -- (28.0, 26.0);
\draw[->] (0.0, 0.0) -- (26.133, 28.0);
\draw[->] (0.0, 0.0) -- (23.059, 28.0);
\draw[->] (0.0, 0.0) -- (20.632, 28.0);
\draw[->] (0.0, 0.0) -- (17.043, 28.0);
\draw[->] (0.0, 0.0) -- (15.68, 28.0);
\draw[->] (0.0, 0.0) -- (28.0, 13.067);
\draw[->] (0.0, 0.0) -- (28.0, 14.933);
\draw[->] (0.0, 0.0) -- (28.0, 20.533);
\draw[->] (0.0, 0.0) -- (28.0, 24.267);
\draw[->] (0.0, 0.0) -- (28.0, 26.133);
\draw[->] (0.0, 0.0) -- (26.25, 28.0);
\draw[->] (0.0, 0.0) -- (24.706, 28.0);
\draw[->] (0.0, 0.0) -- (22.105, 28.0);
\draw[->] (0.0, 0.0) -- (19.091, 28.0);
\draw[->] (0.0, 0.0) -- (18.261, 28.0);
\draw[->] (0.0, 0.0) -- (28.0, 12.25);
\draw[->] (0.0, 0.0) -- (28.0, 15.75);
\draw[->] (0.0, 0.0) -- (28.0, 19.25);
\draw[->] (0.0, 0.0) -- (28.0, 22.75);
\draw[->] (0.0, 0.0) -- (28.0, 26.25);
\draw[->] (0.0, 0.0) -- (26.353, 28.0);
\draw[->] (0.0, 0.0) -- (23.579, 28.0);
\draw[->] (0.0, 0.0) -- (21.333, 28.0);
\draw[->] (0.0, 0.0) -- (19.478, 28.0);
\draw[->] (0.0, 0.0) -- (28.0, 11.529);
\draw[->] (0.0, 0.0) -- (28.0, 13.176);
\draw[->] (0.0, 0.0) -- (28.0, 14.824);
\draw[->] (0.0, 0.0) -- (28.0, 16.471);
\draw[->] (0.0, 0.0) -- (28.0, 18.118);
\draw[->] (0.0, 0.0) -- (28.0, 19.765);
\draw[->] (0.0, 0.0) -- (28.0, 21.412);
\draw[->] (0.0, 0.0) -- (28.0, 23.059);
\draw[->] (0.0, 0.0) -- (28.0, 24.706);
\draw[->] (0.0, 0.0) -- (28.0, 26.353);
\draw[->] (0.0, 0.0) -- (26.444, 28.0);
\draw[->] (0.0, 0.0) -- (25.053, 28.0);
\draw[->] (0.0, 0.0) -- (23.8, 28.0);
\draw[->] (0.0, 0.0) -- (22.667, 28.0);
\draw[->] (0.0, 0.0) -- (21.636, 28.0);
\draw[->] (0.0, 0.0) -- (20.696, 28.0);
\draw[->] (0.0, 0.0) -- (28.0, 10.889);
\draw[->] (0.0, 0.0) -- (28.0, 17.111);
\draw[->] (0.0, 0.0) -- (28.0, 20.222);
\draw[->] (0.0, 0.0) -- (28.0, 26.444);
\draw[->] (0.0, 0.0) -- (26.526, 28.0);
\draw[->] (0.0, 0.0) -- (28.0, 11.789);
\draw[->] (0.0, 0.0) -- (28.0, 13.263);
\draw[->] (0.0, 0.0) -- (28.0, 14.737);
\draw[->] (0.0, 0.0) -- (28.0, 16.211);
\draw[->] (0.0, 0.0) -- (28.0, 17.684);
\draw[->] (0.0, 0.0) -- (28.0, 19.158);
\draw[->] (0.0, 0.0) -- (28.0, 20.632);
\draw[->] (0.0, 0.0) -- (28.0, 22.105);
\draw[->] (0.0, 0.0) -- (28.0, 23.579);
\draw[->] (0.0, 0.0) -- (28.0, 25.053);
\draw[->] (0.0, 0.0) -- (28.0, 26.526);
\draw[->] (0.0, 0.0) -- (26.6, 28.0);
\draw[->] (0.0, 0.0) -- (25.333, 28.0);
\draw[->] (0.0, 0.0) -- (28.0, 12.6);
\draw[->] (0.0, 0.0) -- (28.0, 15.4);
\draw[->] (0.0, 0.0) -- (28.0, 18.2);
\draw[->] (0.0, 0.0) -- (28.0, 23.8);
\draw[->] (0.0, 0.0) -- (28.0, 26.6);
\draw[->] (0.0, 0.0) -- (28.0, 10.667);
\draw[->] (0.0, 0.0) -- (28.0, 13.333);
\draw[->] (0.0, 0.0) -- (28.0, 14.667);
\draw[->] (0.0, 0.0) -- (28.0, 17.333);
\draw[->] (0.0, 0.0) -- (28.0, 21.333);
\draw[->] (0.0, 0.0) -- (28.0, 22.667);
\draw[->] (0.0, 0.0) -- (28.0, 25.333);
\draw[->] (0.0, 0.0) -- (28.0, 11.455);
\draw[->] (0.0, 0.0) -- (28.0, 16.545);
\draw[->] (0.0, 0.0) -- (28.0, 19.091);
\draw[->] (0.0, 0.0) -- (28.0, 21.636);
\draw[->] (0.0, 0.0) -- (28.0, 10.957);
\draw[->] (0.0, 0.0) -- (28.0, 12.174);
\draw[->] (0.0, 0.0) -- (28.0, 13.391);
\draw[->] (0.0, 0.0) -- (28.0, 14.609);
\draw[->] (0.0, 0.0) -- (28.0, 15.826);
\draw[->] (0.0, 0.0) -- (28.0, 17.043);
\draw[->] (0.0, 0.0) -- (28.0, 18.261);
\draw[->] (0.0, 0.0) -- (28.0, 19.478);
\draw[->] (0.0, 0.0) -- (28.0, 20.696);
\draw[->] (0.0, 0.0) -- (28.0, 12.833);
\draw[->] (0.0, 0.0) -- (28.0, 15.167);
\draw[->] (0.0, 0.0) -- (28.0, 12.32);
\draw[->] (0.0, 0.0) -- (28.0, 13.44);
\draw[->] (0.0, 0.0) -- (28.0, 14.56);
\draw[->] (0.0, 0.0) -- (28.0, 15.68);
\draw[->] (0.0, 0.0) -- (28.0, 11.846);
\draw[->] (0.0, 0.0) -- (28.0, 11.407);
\draw[->] (0.0, 0.0) -- (28.0, 13.481);
\draw[->] (0.0, 0.0) -- (28.0, 11.0);
\end{tikzpicture}
\caption{The standard scattering diagrams $\mathfrak{D}^{2,2}$ and $\mathfrak{D}^{3,3}$.}
\label{fig:stand}
\end{figure}

\begin{ex}
\label{expl:scat}
Figure \ref{fig:stand} shows the standard scattering diagrams $\mathfrak{D}^{2,2}$, to $t$-order $10$, and $\mathfrak{D}^{3,3}$, to $t$-order $40$. We omitted $t$ and only show some of the functions $f_{\mathfrak{d}}$. 

Note that $\mathfrak{D}^{2,2}$ has only rays in directions $(1,1)$, $(n,n+1)$ and $(n+1,n)$ for $n\in\mathbb{N}$. The functions $f_{(a,b)}^{2,2}$ are
\[ f_{(1,1)}^{2,2} = (1-xy)^{-4}, \quad f_{(n,n+1)} = (1+x^ny^{n+1})^2, \quad f_{(n+1,n)} = (1+x^{n+1}y^n)^2. \]
Hence, the non-zero coefficients $c_{a,b}^{2,2}$ are
\[ c_{n,n}^{2,2} = \begin{cases} 4, & n=2^k \\ 0, & \text{otherwise} \end{cases}, \quad c_{n,n+1}^{2,2}=2, \quad c_{n+1,n}^{2,2}=2. \]
In particular, the rays are discrete.

For $\mathfrak{D}^{3,3}$ the functions $f_{(a,b)}^{3,3}$ and coefficients $c_{a,b}^{3,3}$ are very complicated and unknown in general. Only for $f_{(1,1)}^{3,3}$ there is a known formula, see \cite[Example 1.6]{GPS}, \cite[\S1.4]{GP}. Note that there seems to be a region that is densely filled with rays. This is exactly the statement of Theorem \ref{thm:main}.
\end{ex}

\subsection{The change of lattice trick}

There is a useful way to reduce to only needing to consider standard diagrams, found in \cite[Proof of Proposition C.13, Step IV]{GHKK}.

\begin{prop}
\label{prop:col}
Let $\mathfrak{D}$ be the consistent diagram obtained from the scattering diagram consisting of two lines $\mathfrak{d}_1$ and $\mathfrak{d}_2$ with functions $f_1=(1+tz^{m_1})^{d_1}$ and $f_2=(1+tz^{m_2})^{d_2}$. Let $M'\subseteq M$ be the sublattice generated by $m_1$ and $m_2$ and let $N' \supseteq N$ be the dual lattice. If $m\in M\setminus M'$ then $f_m=1$. Otherwise write $m=am_1+bm_2$. Then 
\[ f_m^{\mathfrak{D}} = \left(f_{(a,b)}^{d_1 e(m_2^*), d_2 e(m_1^*)}\right)^{1/e(n)} \]
where $n \in N'$ is orthogonal to $m\in M'$ and primitive, and for any $n' \in N'$ we defined
\[ e(n'):=\text{min}\{k \in \mathbb{N} \mid kn' \in N\} \] 
In particular, the scattering of any scattering diagram consisting of two lines can be computed from a standard scattering diagram.
\end{prop}

\begin{proof}
Any ray $\mathfrak{d}$ in $\mathfrak{D}$ has direction vector $m_{\mathfrak{d}}$ contained in $M'\subseteq M$. Hence, we can consider $\mathfrak{d}$, $\mathfrak{D}$ and $\mathfrak{D}_\infty$ in the lattice $M$ or in the lattice $M'$. In the latter case we will write $\mathfrak{d}'$, $\mathfrak{D}'$ and $\mathfrak{D}'_\infty$. By definition the automorphism $\theta_{\mathfrak{d}}\in\text{Aut}_{\mathbb{C}\llbracket t\rrbracket}(\mathbb{C}[M]\widehat{\otimes}_{\mathbb{C}}\mathbb{C}\llbracket t\rrbracket)$ defined by a ray $\mathfrak{d}\in\mathfrak{D}_\infty$ is given by
\[ \theta_{\mathfrak{d}} : z^{m} \mapsto z^{m} f_{\mathfrak{d}}^{\braket{m,n_{\mathfrak{d}}}}. \]
Let $e(n')$ be the defined as above. Then we have $n_{\mathfrak{d}}=e(n_{\mathfrak{d}'})n_{\mathfrak{d}'}$ and the corresponding automorphism $\theta_{\mathfrak{d}'}\in\text{Aut}_{\mathbb{C}\llbracket t\rrbracket}(\mathbb{C}[M']\widehat{\otimes}_{\mathbb{C}}\mathbb{C}\llbracket t\rrbracket)$ defined by $\mathfrak{d}'\in\mathfrak{D}'_\infty$ is given by
\[ \theta_{\mathfrak{d}'} : z^{m'} \mapsto z^{m'}f_{\mathfrak{d}}^{\braket{m',n_{\mathfrak{d}}}} = z^{m'}f_{\mathfrak{d}}^{e(n_{\mathfrak{d}'})\braket{m',n_{\mathfrak{d}'}}} = z^{m'} f_{\mathfrak{d}'}^{\braket{m',n_{\mathfrak{d}'}}}. \]
This shows that $f_{\mathfrak{d}'}=f_{\mathfrak{d}}^{e(n_{\mathfrak{d}'})}$. In particular the initial functions $f_1$ and $f_2$ considered in the lattice $M'$ are $f'_1=(1+tx)^{d_1 e(m_2^*)}$ and $f'_2=1+ty^{d_2 e(m_1^*)}$, where $x=z^{m_1}$ and $y=z^{m_2}$. These are the initial functions of the standard scattering diagram $\mathfrak{D}^{d_1 e(m_2^*), d_2 e(m_1^*)}$. 

We know that scattering gives a consistent diagram $\mathfrak{D}_\infty^{d_1 e(m_2^*), d_2 e(m_1^*)}$. We get a consistent diagram containing $\mathfrak{D}$ by replacing any ray $\mathfrak{d}' \in \mathfrak{D}_\infty^{e(n_1),e(n_2)}$ by $\mathfrak{d}$ with function $f_{\mathfrak{d}}=f_{\mathfrak{d}'}^{1/e(n_{\mathfrak{d}'})}$. By uniqueness of consistent diagrams up to equivalence (Remark \ref{rem:unique}) this completes the proof.
\end{proof}

\begin{ex}
Let $\mathfrak{D}_{\text{exp}}^{\mu,\nu}$ be the scattering diagram consisting of two lines with functions $f_1=1+tx^\mu$ and $f_2=1+ty^\nu$. As above, let $M'$ be the sublattice generated by $m_1=(\mu,0)$ and $m_2=(0,\nu)$. Then $e(m_1^*)=\mu$ and $e(m_2^*)=\nu$, so $\mathfrak{D}_{\text{exp}}^{\mu,\nu}$ is related to $\mathfrak{D}^{\nu,\mu}$. Let $m\in M'$ be primitive in $M'$ with dual $n\in N'$ and write $m = am_1+bm_2$. Then $e(n)=\mu\nu/\gcd(m_{(1)},m_{(2)})$, where $m_{(i)}$ is the $i$-th component of $m\in M$. So we have
\[ f_m^{\mathfrak{D}_{\text{exp}}^{\mu,\nu}} = \left(f_{(a,b)}^{\nu,\mu}\right)^{\gcd(m_{(1)},m_{(2)})/\mu\nu} = \left(f_{(b,a)}^{\mu,\nu}\right)^{\gcd(m_{(1)},m_{(2)})/\mu\nu}. \]
Figure \ref{fig:exp} shows $\mathfrak{D}_{\text{exp}}^{\mu,\nu}$ and $\mathfrak{D}^{\nu,\mu}$ for $(\mu,\nu)=(3,2)$ to order $4$.
\end{ex}

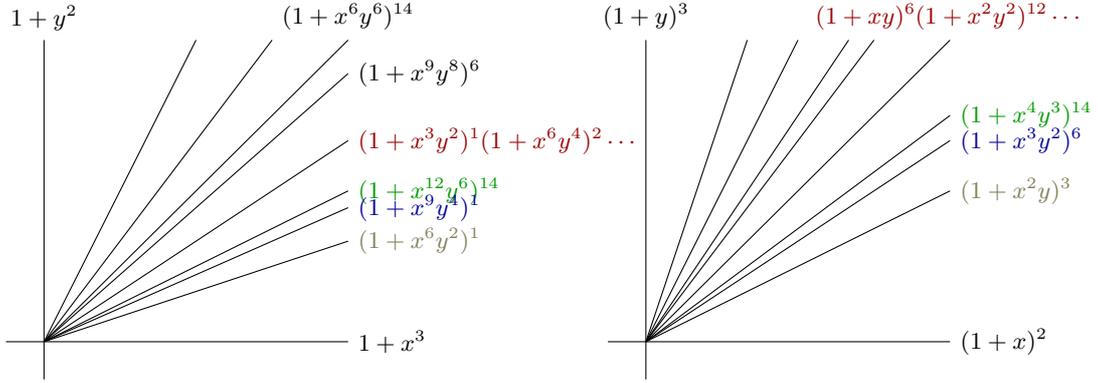
\begin{figure}[h!]
\centering
\makebox[\textwidth]{
\begin{tikzpicture}[scale=1]
\draw[-] (0.0, -1/2) -- (0.0, 4) node[above]{\footnotesize$1+y^2$};
\draw[-] (-1/2, 0.0) -- (4, 0.0) node[right]{\footnotesize$1+x^3$};
\draw[-] (0.0, 0.0) -- (4,4) node[above]{\footnotesize$(1+x^6y^6)^{14}$};
\draw[-] (0.0, 0.0) -- (2, 4);
\draw[-] (0.0, 0.0) -- (4, 2) node[right]{\footnotesize\textcolor{black!40!green}{$(1+x^{12}y^6)^{14}$}};
\draw[-] (0.0, 0.0) -- (4, 4/3) node[right]{\footnotesize\textcolor{black!60!yellow}{$(1+x^6y^2)^1$}};
\draw[-] (0.0, 0.0) -- (4, 8/3) node[right]{\footnotesize\textcolor{black!40!red}{$(1+x^3y^2)^1(1+x^6y^4)^2\cdots$}};
\draw[-] (0.0, 0.0) -- (3, 4);
\draw[-] (0.0, 0.0) -- (4, 16/9) node[right]{\footnotesize\textcolor{black!40!blue}{$(1+x^9y^4)^1$}};
\draw[-] (0.0, 0.0) -- (4, 32/9) node[right]{\footnotesize$(1+x^9y^8)^6$};
\end{tikzpicture}
\hspace{-10mm}
\begin{tikzpicture}[scale=1]
\draw[-] (-1/2, 0.0) -- (4.0, 0.0) node[right]{\footnotesize$(1+x)^2$};
\draw[-] (0.0, -1/2) -- (0.0, 4.0) node[above]{\footnotesize$(1+y)^3$};
\draw[-] (0.0, 0.0) -- (4.0, 4.0) node[above]{\footnotesize\textcolor{black!40!red}{$(1+xy)^6(1+x^2y^2)^{12}\cdots$}};
\draw[-] (0.0, 0.0) -- (4.0, 2.0) node[right]{\footnotesize\textcolor{black!60!yellow}{$(1+x^2y)^3$}};
\draw[-] (0.0, 0.0) -- (2.0, 4.0);
\draw[-] (0.0, 0.0) -- (4.0, 2.667) node[right]{\footnotesize\textcolor{black!40!blue}{$(1+x^3y^2)^6$}};
\draw[-] (0.0, 0.0) -- (1.333, 4.0);
\draw[-] (0.0, 0.0) -- (2.667, 4.0);
\draw[-] (0.0, 0.0) -- (4.0, 3.0) node[right]{\footnotesize\textcolor{black!40!green}{$(1+x^4y^3)^{14}$}};
\draw[-] (0.0, 0.0) -- (3.0, 4.0);
\end{tikzpicture}
}
\caption{The diagrams $\mathfrak{D}_{\text{exp}}^{3,2}$ and $\mathfrak{D}^{2,3}$ to order $4$.}
\label{fig:exp}
\end{figure}

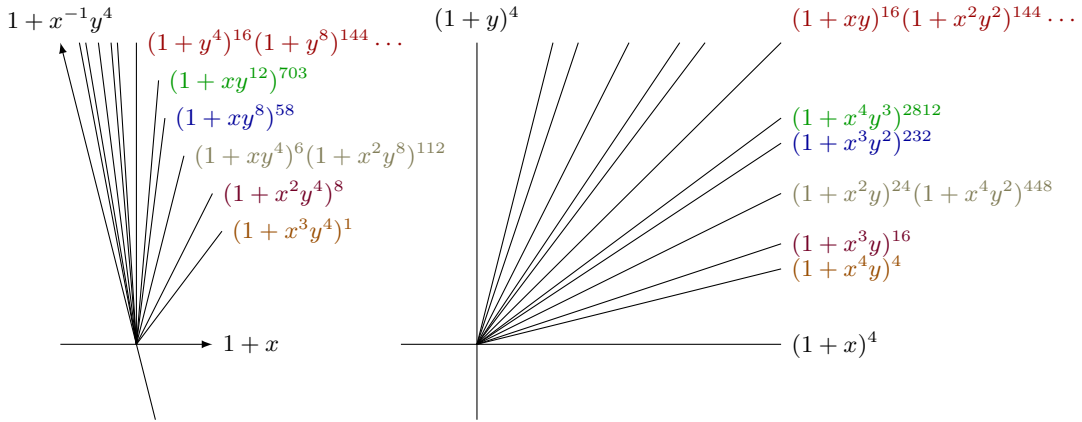
\begin{figure}[h!]
\centering
\makebox[\textwidth]{
\begin{tikzpicture}[scale=1]
\draw[-] (0.0, 0.0) -- (-3/4, 4);
\draw[->] (1/4, -1) -- (-1.0, 4.0) node[above]{\footnotesize$1+x^{-1}y^4$};
\draw[-] (0.0, 0.0) -- (-2/3, 4);
\draw[-] (0.0, 0.0) -- (-1/2, 4);
\draw[-] (0.0, 0.0) -- (-1/3, 4);
\draw[-] (0.0, 0.0) -- (-1/4, 4);
\draw[-] (0.0, 0.0) -- (0.0, 4) node[right]{\footnotesize\textcolor{black!40!red}{$(1+y^4)^{16}(1+y^8)^{144}\cdots$}};
\draw[->] (-1.0, 0.0) -- (1.0, 0.0) node[right]{\footnotesize$1+x$};
\draw[-] (0.0, 0.0) -- (7/24, 7/2) node[right]{\footnotesize\textcolor{black!40!green}{$(1+xy^{12})^{703}$}};
\draw[-] (0.0, 0.0) -- (3/8, 3) node[right]{\footnotesize\textcolor{black!40!blue}{$(1+xy^8)^{58}$}};
\draw[-] (0.0, 0.0) -- (5/8, 5/2) node[right]{\footnotesize\textcolor{black!60!yellow}{$(1+xy^4)^6(1+x^2y^8)^{112}$}};
\draw[-] (0.0, 0.0) -- (1, 2) node[right]{\footnotesize\textcolor{black!40!purple}{$(1+x^2y^4)^8$}};
\draw[-] (0.0, 0.0) -- (9/8, 3/2) node[right]{\footnotesize\textcolor{black!40!orange}{$(1+x^3y^4)^1$}};
\end{tikzpicture}
\hspace{-10mm}
\begin{tikzpicture}[scale=1]
\draw[-] (0.0, -1.0) -- (0.0, 4.0) node[above]{\footnotesize$(1+y)^4$};
\draw[-] (-1.0, 0.0) -- (4.0, 0.0) node[right]{\footnotesize$(1+x)^4$};
\draw[-] (0.0, 0.0) -- (4.0, 4.0) node[above right]{\footnotesize\textcolor{black!40!red}{$(1+xy)^{16}(1+x^2y^2)^{144}\cdots$}};
\draw[-] (0.0, 0.0) -- (2.0, 4.0);
\draw[-] (0.0, 0.0) -- (1.333, 4.0);
\draw[-] (0.0, 0.0) -- (1.0, 4.0);
\draw[-] (0.0, 0.0) -- (4.0, 2.0) node[right]{\footnotesize\textcolor{black!60!yellow}{$(1+x^2y)^{24}(1+x^4y^2)^{448}$}};
\draw[-] (0.0, 0.0) -- (2.667, 4.0);
\draw[-] (0.0, 0.0) -- (4.0, 1.333) node[right]{\footnotesize\textcolor{black!40!purple}{$(1+x^3y)^{16}$}};
\draw[-] (0.0, 0.0) -- (4.0, 2.667) node[right]{\footnotesize\textcolor{black!40!blue}{$(1+x^3y^2)^{232}$}};
\draw[-] (0.0, 0.0) -- (3.0, 4.0);
\draw[-] (0.0, 0.0) -- (4.0, 1.0) node[right]{\footnotesize\textcolor{black!40!orange}{$(1+x^4y)^4$}};
\draw[-] (0.0, 0.0) -- (4.0, 3.0) node[right]{\footnotesize\textcolor{black!40!green}{$(1+x^4y^3)^{2812}$}};
\end{tikzpicture}
}
\caption{The diagrams $\mathfrak{D}_{\text{det}}^4$ and $\mathfrak{D}^{4,4}$ to order $4$.}
\label{fig:det}
\end{figure}

\begin{ex}
\label{ex:det}
Let $\mathfrak{D}_{\text{det}}^k$ be the scattering diagram consisting of two lines with functions $f_1=1+tz^{m_1}$ and $f_2=tz^{m_2}$ such that $m_1$ and $m_2$ are primitive and $|\text{det}(m_1,m_2)|=k$. By an isomorphism of lattices we can bring this into the form $f_1=1+tx$ and $f_2=1+tx^{-1}y^k$. Let $M'$ be the sublattice generated by $m_1=(1,0)$ and $m_2=(-1,k)$ and consider $m = am_1+bm_2 \in M'$ primitive with dual $n\in N'$. Then
\[ f_m^{\mathfrak{D}_{\text{det}}^k} = (f_{(a,b)}^{k,k})^{\text{gcd}(k,m_{(1)})/k}, \]
where $m_{(1)}$ is the first component of $m\in M$. This is because we have $e(m_1^*)=e(m_2^*)=k$ and $e(n)=\text{gcd}(k,m_{(1)})/k$. Figure \ref{fig:det} shows $\mathfrak{D}_{\text{det}}^k$ and $\mathfrak{D}^{k,k}$ for $k=4$ to order $4$.
\end{ex}

\subsection{Deformations}

Given a consistent scattering diagram \(\mathfrak D\), we can form the asymptotic diagram \(\mathfrak D_{\text{as}}\) by replacing every ray \((b_{\mathfrak d} + \mathbb R_{\geq 0} m_{\mathfrak d}, f_{\mathfrak d})\) with \((\mathbb R_{\geq 0} m_{\mathfrak d}, f_{\mathfrak d})\), and similarly for lines. By considering sufficiently large curves in \(\mathfrak D\) around the origin containing all singularities, we see that \(\mathfrak D_{\text{as}}\) is also consistent. We can use this to consider \emph{deformations} as follows. For more details see \cite[\S1.4]{GPS} and \cite[Proposition C.13, Step III]{GHKK}.

\begin{dfn}
The \textit{full deformation} of $\mathfrak{D}^{\mu,\nu}$ consists of general lines $\mathfrak{d}_{1,1},\ldots,\mathfrak{d}_{1,\mu}$ and $\mathfrak{d}_{2,1},\ldots,\mathfrak{d}_{2,\nu}$ with functions 
\[ f_{\mathfrak{d}_{1,i}} = 1+tx, \quad f_{\mathfrak{d}_{2,i}} = 1+ty. \]
Here the lines being general means that all rays of rational slope in the consistent diagram intersect in points, not in rays. We will also consider \textit{partial deformations} obtained by pulling only one factor out, as in Figure \ref{fig:pdef31}.
\end{dfn}

\begin{prop}[\cite{GPS},\S1.4]
Let $\mathfrak{D}'$ be a partial or full deformation of $\mathfrak{D}$. Then
\[ (\mathfrak{D}'_\infty)_{\text{as}} = \mathfrak{D}_\infty. \]
\end{prop}

\begin{ex}
    We can use a partial deformation of the diagram \(\mathfrak D^{3, 1}\) to compute \(\mathfrak D_\infty^{3, 1}\), by performing scattering at each singularity individually. We obtain the scattering diagram shown in Figure \ref{fig:pdef31}, so taking the asymptotic diagram we get \(\mathfrak D^{3, 1}_\infty\).
    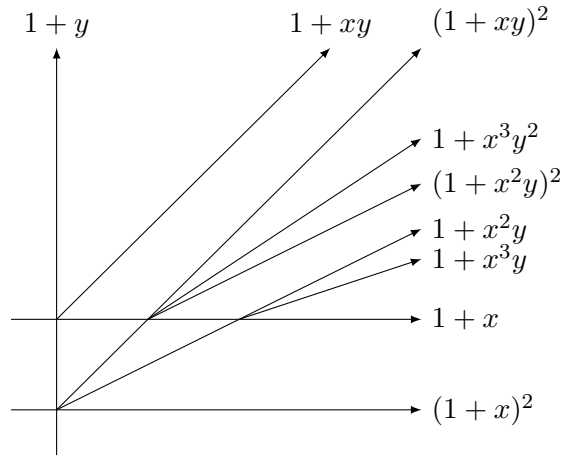
\begin{figure}[h!]
    \centering
    \begin{tikzpicture}[scale=1.2]
    \draw[->] (0.0, -1/2) -- (0.0, 4.0) node[above]{\(1 + y\)};
    \draw[->] (-1/2, 0.0) -- (4.0, 0.0) node[right]{\((1 + x)^2\)};
    \draw[->] (-1/2, 1.0) -- (4.0, 1.0) node[right]{\(1 + x\)};
    \draw[->] (0.0, 0.0) -- (4.0, 4.0) node[above right]{\((1 + x y)^2\)};
    \draw[->] (0.0, 1.0) -- (3.0, 4.0) node[above]{\(1 + x y\)};
    \draw[->] (0.0, 0.0) -- (4.0, 2.0) node[right]{\(1 + x^2 y\)};
    \draw[->] (1.0, 1.0) -- (4.0, 2.5) node[right]{\((1 + x^2 y)^2\)};
    \draw[->] (1.0, 1.0) -- (4.0, 3.0) node[right]{\(1 + x^3 y^2\)};
    \draw[->] (2.0, 1.0) -- (4.0, 1.6666) node[right]{\(1 + x^3 y\)};
    \end{tikzpicture}
    \caption{Computing \(\mathfrak D_\infty^{3, 1}\) using partial deformation.}
    \label{fig:pdef31}
    \end{figure}
\end{ex}

\begin{lem}
If \(\mu \leq \mu'\) and \(\nu \leq \nu'\), \[c_{a, b}^{\mu, \nu} \leq c_{a, b}^{\mu', \nu'} .\]
\end{lem}
\begin{proof}
Deform \(\mathfrak D^{\mu', \nu'}\) in a way so that we have a horizontal line with function \((1 + x)^\mu\) and a vertical line with function \((1 + y)^\nu\). From this we get a ray contributing \(c_{a, b}^{\mu, \nu}\). As all coefficients are positive by Proposition \ref{prop:pos}, we see that \(c_{a, b}^{\mu, \nu} \leq c_{a, b}^{\mu', \nu'}\).
\end{proof}

\subsection{Relation with curve counting}
\label{S:GW}

A scattering diagram is naturally the support of a \emph{tropical curve} (see \cite[Definition 2.1]{GPS}). There is a bijective correspondence between rays in the diagram and tropical curves supported on the diagram, and the coefficient of the function attached to a ray equals the multiplicity of the tropical curve (\cite[Theorem 2.4]{GPS}). Together with a tropical correspondence theorem (\cite[Theorem 3.4]{GPS}) and the degeneration formula (\cite[\S5.3]{GPS}) this leads to the following correspondence, which is a special case of \cite[Theorem 5.4]{GPS}.

Let $m=(a,b) \in M$ be a primitive direction vector. Let $Y_{a,b}^{\mu,\nu}$ be the blow up of the weighted projective plane $\mathbb{P}_{x_0,x_1,x_2}(1,a,b)$ in $\mu$ general points on the degree $a$ divisor $D_1=\{x_1=0\}$ and $\nu$ general points on the degree $b$ divisor $D_2=\{x_2=0\}$. Let $k>0$, let $\mathbf{P}_1,\mathbf{P}_1$ be ordered partitions, possibly with zeros, of $ka,kb$ of lengths $\mu,\nu$, respectively. Let $H$ be the pull back of the hyperplane class of $\mathbb{P}_{x_0,x_1,x_2}(1,a,b)$ to $Y_{a,b}^{\mu,\nu}$, and let $E_{ij}$ be the class of the exceptional divisors, with $i=1,2$ corresponding to blow ups of points on $D_1$ and $D_2$, respectively.

\begin{dfn}
Let $N_{ka,kb}[\mathbf{P}_1,\mathbf{P}_2]$ be the logarithmic \cite{GSloggw} (or relative \cite{Li}) Gromov-Witten invariant counting rational curves on $Y_{a,b}^{\mu,\nu}$ of class $\beta=kH-\sum_{j=1}^\mu p_{1j}E_{1j}-\sum_{j=1}^\nu p_{2j}E_{2j}$ which intersect $D_0=\{x_0=0\}$ in a single unspecified point with maximal tangency. This is a finite number, since $(Y_{a,b},D_0)$ is a log Calabi-Yau pair. But in general $N_{a,b}\in\mathbb{Q}$ is not integral, because there are virtual contributions to the count coming from non-trivial automorphisms.
\end{dfn}

\begin{prop}[{\cite[Theorem 5.4]{GPS}}]
\label{prop:GPS}
Let $\mathfrak{D}$ be the scattering diagram for the ring $\mathbb{C}[x^{\pm 1},y^{\pm 1}]\llbracket t_{1,1},\ldots,t_{1,\mu},t_{2,1}\ldots,t_{2,\nu}\rrbracket$ consisting of two lines with functions $f_1=\prod_{j=1}^\mu (1+t_{1,j}x)$ and $f_2=\prod_{j=1}^\nu (1+t_{2,j}y)$. Then
\[ \log f_{(a,b)}^{\mathfrak{D}} = \sum_{k>0}\sum_{(\mathbf{P}_1,\mathbf{P}_2)} k N_{ka,kb}[\mathbf{P}_1,\mathbf{P}_2]\left(\prod_{j=1}^\mu t_{1,j}^{p_{1,j}}\right)\left(\prod_{j=1}^\nu t_{2,j}^{p_{2,j}}\right)x^{ak}y^{bk}, \]
where the second sum is over pairs of ordered partitions of $ka$ and $kb$ of lengths $\mu$ and $\nu$, respectively.
\end{prop}

Note that by symmetry of the blow ups the numbers $N_{ka,kb}[\mathbf{P}_1,\mathbf{P}_2]$ do not depend on the ordering. Hence we can instead consider unordered partitions (without zeros) $\vec{p}_1$ and $\vec{p}_2$ of $ka$ and $kb$ of lengths at most $\mu$ and $\nu$, respectively.

\begin{dfn}
\label{dfn:N}
Define
\[ N_{\vec{p}_1,\vec{p}_2} := N_{ka,kb}[\mathbf{P}_1,\mathbf{P}_2]. \]
for any ordered partitions $\mathbf{P}_1,\mathbf{P}_2$ with the same elements as $\vec{p}_1,\vec{p}_2$, filled up with zeros.

This is the number of curves with prescribed intersection pattern with the exceptional divisors, but where we forget the labeling of the divisors.
\end{dfn}

\begin{dfn}
For a partition $\vec{p}$, let $\ell(\vec{p})$ be the tuple whose entries are the number of times with which the entries of $\vec{p}$ occur.
\end{dfn}

\begin{cor}
\label{cor:GPS}
We have
\[ \log f_{a,b}^{\mu,\nu} = \sum_{k>0}k N_{ka,kb}^{\mu,\nu}t^{ka+kb}x^{ka}y^{kb}, \]
where
\[ N_{ka,kb}^{\mu,\nu} = \sum_{(\vec{p}_1,\vec{p}_2)} \binom{\mu}{\ell(\vec{p}_1)}\binom{\nu}{\ell(\vec{p}_2)}N_{\vec{p}_1,\vec{p}_2}, \]
where the sum is over pairs of unordered partitions $\vec{p}_1$ and $\vec{p}_2$ of $ka$ and $kb$ of lengths at most $\mu$ and $\nu$, respectively, and the coefficients are multinomial coefficients.
\end{cor}

\begin{proof}
This follows from Proposition \ref{prop:GPS} by replacing all $t_{i,j}$ by the same $t$.
\end{proof}

\subsection{Relation with quiver representations}
\label{S:quiv}

We recall the basics of quiver representations from \cite{Kin}\cite{Rei1} and of bipartite quivers from \cite{RW}.

\begin{dfn}
A \emph{quiver} $Q$ is a directed multigraph, defined by a set $Q_0$ of vertices and a set $Q_1$ of arrows $\alpha : i \rightarrow j$. A \emph{quiver representation} of $Q$ with dimension vector $\mathbf{d}=(d_i)\in\mathbb{N}^{|Q_0|}$ is a collection of vector spaces $V_i$ of dimension $d_i$ and, for each arrow $\alpha : i \rightarrow j$, a linear map $V_i \rightarrow V_j$.

For $\Theta : \mathbb{Z}^{|Q_0|} \rightarrow \mathbb{Z}$ a representation is $\Theta$-stable if for each proper subrepresentation the dimension vector $\mathbf{d}'$ satisfies
\[ \rho_\Theta(\mathbf{d}') := \frac{\Theta(\mathbf{d}')}{\sum_i d'_i} < \rho_\Theta(\mathbf{d}) := \frac{\Theta(\mathbf{d})}{\sum_i d_i}. \]
There exists a moduli space $M_{\mathbf{d}}^{\Theta-st}(Q)$ of isomorphism classes of $\Theta$-stable representations of $Q$ with dimension vector $\mathbf{d}$. If non-empty, it is a smooth irreducible variety of dimension $1-\braket{\mathbf{d},\mathbf{d}}$, where we used the Euler form
\[ \braket{\mathbf{d},\mathbf{e}} = \sum_i d_ie_i - \sum_{\alpha : i \rightarrow j} d_ie_j. \]
\end{dfn}

\begin{ex}
The $\mu$-Kronecker quiver $Q_\mu$ is the quiver with $2$ vertices and $\mu$ arrows, all from one to the other. The complete bipartite quiver $Q_{\mu,\nu}$ has $\mu+\nu$ vertices and $\mu\nu$ arrows, one from each of the first $\mu$ to each of the last $\nu$.
\end{ex}

\begin{dfn}
\label{dfn:bipartite}
Define
\[ M_{(a,b)}^{\Theta-st}(Q_{\mu,\nu}) := \coprod_{\mathbf{P}_1,\mathbf{P}_2} M_{(\mathbf{P}_1,\mathbf{P}_2)}^{\Theta-st}(Q_{\mu,\nu}), \]
where the coproduct is over all pairs $\mathbf{P}_1,\mathbf{P}_2$ of ordered partitions, possibly with zeros, of $a,b$ of lengths $\mu,\nu$, respectively.
\end{dfn}

\begin{prop}
We have
\[ M_{(d_1,d_2)}^{\Theta-st}(Q_\mu) \simeq M_{(d_1,d_2)}^{\Theta-st}(Q_{\mu,\mu}). \]
In particular, this induces a decomposition as above for $M_{(d_1,d_2)}^{\Theta-st}(Q_\mu)$.
\end{prop}

\begin{prop}[{\cite[Theorem 5.1]{RW}}]
\label{prop:irred}
If $M_{(\mathbf{P}_1,\mathbf{P}_2)}^{\Theta-st}(Q_{\mu,\nu})$ is non-empty, then it is a smooth irreducible variety of dimension
\[ 1-\braket{(\mathbf{P}_1,\mathbf{P}_2),(\mathbf{P}_1,\mathbf{P}_2)} = 1 - \sum_{j=1}^\mu p_{1j}^2 - \sum_{j=1}^\nu p_{2j}^2 + ab. \]
\end{prop}

To make the relation with scattering diagrams we need the notion of framed quiver representations.

\begin{dfn}
\label{dfn:framed}
A front framed (resp. back framed) respresentation of the $\mu$-Kronecker quiver $Q_\mu$ is a representation of $Q_\mu$ together with the choice of $1$-dimensional subspace $L\subset V_1$ (resp. $L\subset V_2$) of the vector space corresponding to the first (resp. second) vertex of $Q_\mu$. One similarly defines framed representations of bipartite quivers $Q_{\mu,\nu}$.
\end{dfn}

\begin{dfn}
A framed representation is $\Theta$-stable if is $\Theta$-semistable with respect to all proper subrepresentations and $\Theta$-stable with respect to proper subrepresentations containing $L$.

We denote the moduli space of front framed (resp. back framed) $\Theta$-stable representations of $Q_{\mu,\nu}$ with dimension vector $(d_1,d_2)$ by $M_{(d_1,d_2)}^{\Theta-st,F}(Q_{\mu,\nu})$ (resp. $M_{(d_1,d_2)}^{\Theta-st,B}(Q_{\mu,\nu})$).
\end{dfn}

By \cite[Lemma 3.2]{Rei1} all stability conditions are equivalent to stability with respect to $\Theta$-stability for $\Theta\in\{(1,0),(1,1),(0,1)\}$. It turns out that scattering diagrams are related to $(1,0)$-stability.

\begin{prop}[\cite{RW}, Theorem 6.1]
We have
\[ f_{(a,b)}^{\mu,\nu} = (B_{(a,b)})^{\frac{\mu}{a}} = (f_{(a,b)})^{\frac{\nu}{b}}, \]
where
\begin{align*}
B_{(a,b)} &= 1+\sum_{k>0}\chi(M_{(ak,bk)}^{(1,0)-st,B}(Q_{\mu,\nu}))t^{k(a+b)}x^ay^b, \\
F_{(a,b)} &= 1+\sum_{k>0}\chi(M_{(ak,bk)}^{(1,0)-st,F}(Q_{\mu,\nu}))t^{k(a+b)}x^ay^b. \\
\end{align*}
\end{prop}

The relation between $f_{(a,b)}$ and the Euler characteristic of non-framed quiver representations is coming from a system of functional equations and is quite complicated in general, see \cite{RW}, Theorem 8.1 and Corollary 8.2. But for the first term, i.e., for primitive direction vectors $(a,b)$, this reduces to the following simple relation.

\begin{dfn}
Define
\[ \chi_{a,b}^{\mu,\nu} := \chi(M_{(a,b)}^{(1,0)-st}(Q_{\mu,\nu})). \]
\end{dfn}

\begin{prop}[\cite{RW}, Corollary 9.1]
\label{prop:chi}
If $(a,b)\in\mathbb{Z}_{>0}^2$ is primitive, then
\[ c_{a,b}^{\mu,\nu} = \chi_{a,b}^{\mu,\nu}. \]
\end{prop}

\subsection{Mutations}
\label{S:mut}

\begin{dfn}
For $\mu, \nu \in \mathbb Z_{> 0}$ define two \emph{mutation} actions on $\mathbb{Z}^2$ by
\begin{align*}
\mathbf T_1^{\mu,\nu} : (a,b) \mapsto \begin{cases} (\mu b - a, b), & b > 0, \\ (a,b), & b \leq 0, \end{cases} \\
\mathbf T_2^{\mu,\nu} : (a,b) \mapsto \begin{cases} (a, \nu a - b), & a > 0, \\ (a,b), & a \leq 0. \end{cases}
\end{align*}
Here $\mathbb{Z}^2$ will be the space of direction vectors $(a,b)$ of rays in a scattering diagram, and $(a,b)$ will usually be assumed to be primitive, i.e. $\text{gcd}(a,b)=1$, and such that $(a,b)$, $\mathbf{T}_1^{\mu,\nu}(a,b)$ and $\mathbf{T}_2^{\mu,\nu}(a,b)$ are all contained in the first quadrant $\mathbb{Z}_{\geq 0}^2$.
\end{dfn}

\begin{prop}[{\cite[Theorem 1.24]{GHKK}}]
We have
\[ f_{m}^{\mu,\nu}=f_{\mathbf T_1^{\mu,\nu}(m)}^{\mu,\nu}=f_{\mathbf T_2^{\mu,\nu}(m)}^{\mu,\nu}. \]
\end{prop}

In the case $\mu=\nu$ this follows from the correspondence between scattering diagrams and quiver representations (\cite{Rei2}, see also \cite[Theorem 1]{GP}) and reflection functors for quivers (\cite{BGP}, see also \cite[\S5.3]{GP}). For $\mu \neq \nu$ there might be a similar proof using bipartite quivers (\cite{RW}). The statement also has also been proved for general $\mu,\nu$ using the relation with Gromov-Witten invariants (\cite[Theorem 5.4]{GPS}) and equivalence of log Calabi-Yau pairs under certain birational transformations (\cite[Theorem 7]{GP}). Note that \cite[Theorem 7]{GP}, is only stated in the case $a,b>0$, but it easily generalizes to $a=0$ or $b=0$ by considering classes of a fiber minus an exceptional line. In \cite{GHKK} the statement was proved directly from the definition of scattering diagrams, without reference to quivers or Gromov-Witten invariants.

\begin{dfn}
The mutations naturally act on $\mathbb{Q}_{>0}$, where $\rho=\frac{b}{a}\in\mathbb{Q}_{>0}$ can be thought of as the slope of the direction vector $(a,b)\in\mathbb{Z}_{>0}^2$. The action is given by
\begin{align*}
T_1(\rho) &= (\mu - \tfrac{1}{\rho})^{-1} , \\
T_2(\rho) &= \nu - \rho .
\end{align*}
\end{dfn}

\begin{lem}
\label{lem:mutslope}
$T_1$ and $T_2$ are order reversing (or strictly decreasing) for $\frac{1}{\mu} < \rho < \nu$ and have fixed points $\rho_{0,-}^{\mu,\nu}=\frac{2}{\mu}$ and $\rho_{0,+}^{\mu,\nu}=\frac{\nu}{2}$, respectively.
\end{lem}

\begin{proof}
This is clear from the definition.
\end{proof}

\section{The dense region}

Consider a standard scattering diagram $D^{\mu,\nu}$ (Definition \ref{def:standard}). Mutations (\S\ref{S:mut}) act on the directions $\mathbb{Z}^2$ (or slopes $\mathbb{Q}$). They have some fixed points and naturally divide the scattering diagram into certain \textit{regions}. We will show the following. For $\mu\nu>4$ there is a \textit{dense region} $\Phi$ in which every slope occurs with non-trivial function (Theorem \ref{thm:full}). It is made up of an infinite number of fundamental domains $\phi_k$  (for the mutation action). Outside of the dense region there is a discrete number of rays and each of them appears with coefficients $\mu$ or $\nu$, because they come from mutation of the initial rays (Proposition \ref{prop:outside}). All rays produced from scattering have slope $\frac{1}{\mu} < \rho < \nu$ (Proposition \ref{prop:bound}). The situation is summarized in Figure \ref{fig:regions}. The slopes $\rho_{0,\pm}^{\mu,\nu}$ are defined in Lemma \ref{lem:mutslope}, the slopes $\rho_{\pm}^{\mu,\nu}$ will be defined in Definition \ref{def:dense}.

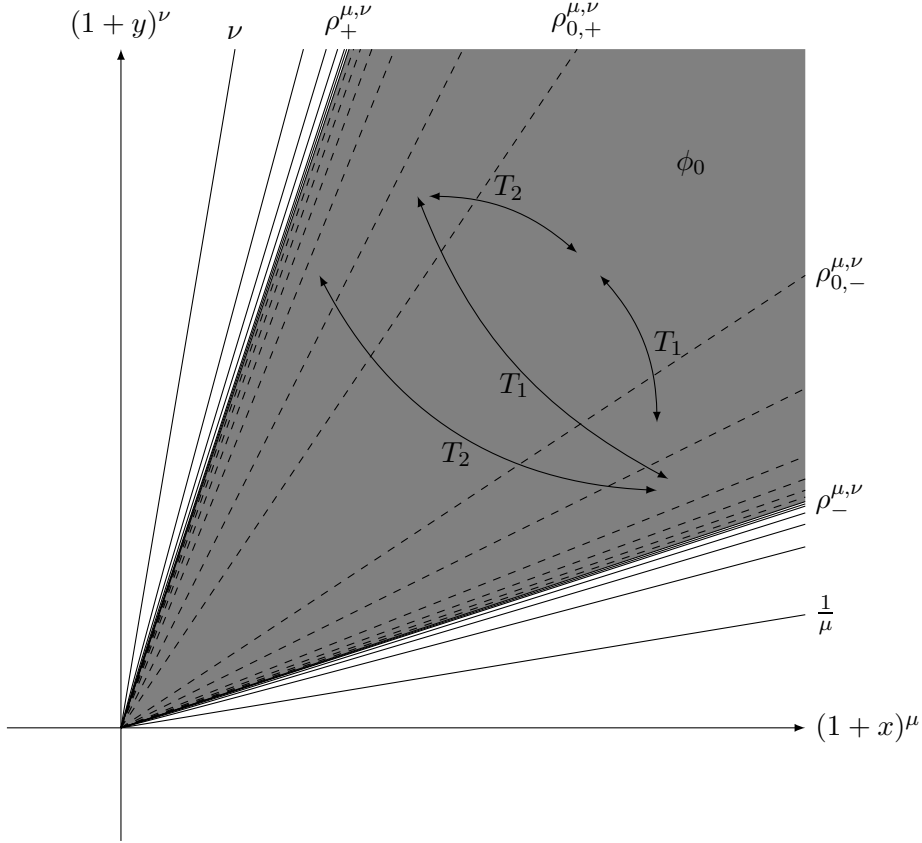
\begin{figure}[h!]
\centering
\begin{tikzpicture}[scale=3]
\fill[color=gray] (0,0) -- (1,3) -- (3,3) -- (3,1) -- (0,0);
\draw[->] (-1/2,0) -- (3,0) node[right]{$(1+x)^\mu$};
\draw[->] (0,-1/2) -- (0,3) node[above]{$(1+y)^\nu$};
\draw (0,0) -- (3,.5) node[right]{$\frac{1}{\mu}$};
\draw (0,0) -- (3,.8);
\draw (0,0) -- (3,.9);
\draw (0,0) -- (3,.95);
\draw (0,0) -- (3,.98);
\draw (0,0) -- (3,.99);
\draw (0,0) -- (.5,3) node[above]{$\nu$};
\draw (0,0) -- (.8,3);
\draw (0,0) -- (.9,3);
\draw (0,0) -- (.95,3);
\draw (0,0) -- (.98,3);
\draw (0,0) -- (.99,3);
\draw[dashed] (0,0) -- (3,2) node[right]{$\rho_{0,-}^{\mu,\nu}$};
\draw[dashed] (0,0) -- (3,1.5);
\draw[dashed] (0,0) -- (3,1.2);
\draw[dashed] (0,0) -- (3,1.1);
\draw[dashed] (0,0) -- (3,1.05);
\draw[dashed] (0,0) -- (3,1.02);
\draw (0,0) -- (3,1) node[right]{$\rho_-^{\mu,\nu}$};
\draw[dashed] (0,0) -- (2,3) node[above]{$\rho_{0,+}^{\mu,\nu}$};
\draw[dashed] (0,0) -- (1.5,3);
\draw[dashed] (0,0) -- (1.2,3);
\draw[dashed] (0,0) -- (1.1,3);
\draw[dashed] (0,0) -- (1.05,3);
\draw[dashed] (0,0) -- (1.02,3);
\draw (0,0) -- (1,3) node[above]{$\rho_+^{\mu,\nu}$};
\draw (2.5, 2.5) node{\(\phi_0\)};
\draw[<->] (2,2.1) to[bend right=20] node[above,midway]{$T_2$} (1.35,2.35);
\draw[<->] (2.1,2) to[bend left=20] node[right,midway]{$T_1$} (2.35,1.35);
\draw[<->] (1.3,2.35) to[bend right=20] node[below,midway]{$T_1$} (2.4,1.1);
\draw[<->] (2.35,1.05) to[bend left=30] node[below,midway]{$T_2$} (0.87, 2);
\end{tikzpicture}
\caption{The regions of the scattering diagram.}
\label{fig:regions}
\end{figure}

\begin{dfn}
For a ray in direction \((a, b)\) lying strictly in the first quadrant, the \emph{slope} of the ray is \(\rho = b/a\).
\end{dfn}

\begin{prop}
\label{prop:bound}
In a standard scattering diagram \(\mathfrak D^{\mu, \nu}\), every ray with direction \((a, b)\in\mathbb{Z}_{>0}^2\) satisfies \[\frac{1}{\mu} \leq \frac{b}{a} \leq \nu .\]
\end{prop}

\begin{proof}
We proceed by induction on \(a + b\), and then on \(\mu + \nu\). Certainly \((a, b) = (1, 1)\) satisfies this condition in all standard scattering diagrams, and when \((\mu, \nu) = (1, 1)\), this is the only ray.

We may assume \(\nu > 1\). Consider a partial deformation of $\mathfrak D^{\mu, \nu}$ to $D^{\mu,\nu-1}$ and $D^{\mu,1}$, as in Figure \ref{fig:pdef31}. By induction, all rays from the \(\mathfrak D^{\mu, 1}\) subdiagram satisfy the condition, so consider some ray from the \(\mathfrak D^{\mu, \nu - 1}\) subdiagram in the direction \((a_0, b_0)\) striking the vertical line to produce a ray in the direction \((a, b)\). From the change of lattice trick, this scattering corresponds to \(\mathfrak D^{a_0 c_{a_0, b_0}^{\mu, \nu}, a_0}\) (see Example \ref{ex:det}), and \((a, b) = \alpha (a_0, b_0) + \beta (0, 1)\). As all rays occur in the first quadrant, \(a_0 + b_0, \alpha + \beta < a + b\), so we use the induction hypothesis to obtain \(\tfrac{1}{\mu} \leq  \tfrac{b_0}{a_0} \leq \nu - 1\) and \( \frac{\beta}{\alpha} \leq a_0\). So \[\frac{1}{\mu} \leq \frac{b_0}{a_0} \leq \frac{b}{a} = \frac{\alpha b_0 + \beta}{\alpha a_0} = \frac{b_0}{a_0} + \frac{\beta}{\alpha a_0} \leq \nu - 1 + \frac{a_0}{a_0} = \nu\] as claimed. If \(\nu = 1\) we either have the base case \((\mu, \nu) = (1, 1)\), or \(\mu > 1\) and we can do the same with a deformation in the other direction to get the result, using that all further scattering between rays is contained between those rays.
\end{proof}

In \(\mathfrak D^{\mu, \nu}\), there are no rays with slope \(0 < \rho < \tfrac{1}{\mu}\) or \(\nu < \rho < \infty\), and the mutations \(T_1^{\mu,\nu}, T_2^{\mu,\nu}\) are order-reversing. Therefore there are no rays with slope \(\nu - \tfrac{1}{\mu} < \rho < \nu\) or \(\tfrac{1}{\mu} < \rho < \tfrac{\nu}{\mu \nu - 1}\). Repeating this we can find a region for which rays cannot be dense outside of, i.e. any dense behaviour must occur inside a ``dense region".

\begin{dfn}
\label{def:dense}
In a standard scattering diagram \(\mathfrak D^{\mu, \nu}\) where \(\mu \nu > 4\), the \emph{dense region} \(\Phi^{\mu, \nu}\) is the cone spanned by the rays from the origin with slope \[\rho_\pm^{\mu, \nu} = \frac{\mu \nu \pm \sqrt{\mu \nu (\mu \nu - 4)}}{2 \mu} = \frac{2 \nu}{\mu \nu \mp \sqrt{\mu \nu (\mu \nu - 4)}} .\] A ray with slope \(\rho\) is in the dense region if \(\rho_-^{\mu, \nu} < \rho < \rho_+^{\mu, \nu}\).
\end{dfn}

\begin{dfn}
\label{def:full}
A cone $\Phi\subset\mathbb{R}_{>0}^2$ is \emph{full} in $\mathfrak{D}^{\mu,\nu}$ if $c_{a,b}^{\mu,\nu}\neq 0$ for every $(a,b)\in\mathbb{Z}_{>0}^2$ (not necessarily primitive) such that the ray in direction $(a,b)$ lies in $\Phi$.
\end{dfn}

In this section we will prove the following statement by induction.

\begin{thm}[Theorem \ref{thm:main}(a)]
\label{thm:full}
\(\Phi^{\mu, \nu}\) is full in $\mathfrak{D}^{\mu,\nu}$ (and in particular dense with rays) when $\mu\nu>4$.
\end{thm}

\subsection{Induction step}

\begin{dfn}
In a standard scattering diagram $D^{\mu,\nu}$, the \emph{fundamental region} \(\phi_0^{\mu, \nu}\) is the cone spanned by the rays \(\mathbb R_{\geq 0} (\mu, 2)\) and \(\mathbb R_{\geq 0} (2, \nu)\). A ray with slope \(\rho\) is in the fundamental region if \(2/\mu \leq \rho \leq \nu/2\).
\end{dfn}

\begin{lem}
\label{lem:phiPhi}
If \(\phi_0^{\mu, \nu}\) is full, then so is \(\Phi^{\mu, \nu}\).
\end{lem}
\begin{proof}
Note that as \(T_1, T_2\) are order-reversing, they map images of the fundamental region to other intervals of slope, with bounds determined by the images of the bounds of the fundamental region. As \(\tfrac{2}{\mu}, \tfrac{\nu}{2} \neq \rho_\pm^{\mu, \nu}\), images of the fundamental region tend toward the bounds of the dense region. Additionally, as \(T_1 (\tfrac{2}{\mu}) = \tfrac{2}{\mu}\) and \(T_2 (\tfrac{\nu}{2}) = \tfrac{\nu}{2}\), there is no gap between these images. Therefore they cover the dense region.

As \(T_1, T_2\) are linear on \(\tfrac{a}{b}\) and \(\tfrac{b}{a}\) respectively, the image of a full region is also full, so if \(\phi_0^{\mu, \nu}\) is full, then so is \(\Phi^{\mu, \nu}\).
\end{proof}

\begin{lem}
Suppose \(\mu (\nu - 1) > 4\) with \(\nu > 2\). If \(\Phi^{\mu, \nu - 1}\) is full in \(\mathfrak D_\infty^{\mu, \nu - 1}\), then \(\Phi^{\mu, \nu}\) is full in \(\mathfrak D_\infty^{\mu, \nu}\).
\end{lem}
\begin{proof}
As \(c_{a, b}^{\mu, \nu} \geq c_{a, b}^{\mu, \nu - 1}\), \(\Phi^{(\mu, \nu - 1)} \subseteq \mathfrak D_\infty^{\mu, \nu}\) is full as a region in \(\mathfrak D_\infty^{\mu, \nu}\).

From the assumptions, we find that either \(4 \leq \mu (\nu - 2)\) or \(\mu (\nu - 2) < 4 < \mu (\nu - 1)\). In this first case \(4 (\nu - 1) < 4 \nu \leq \mu \nu (\nu - 2)\), and in this second case \(\mu = \nu = 3\). In both cases
\begin{align*}
4 (\nu - 1) &< \mu \nu (\nu - 2) = \mu ((\nu - 1)^2 - 1)\\
\iff \mu &< (\nu - 1) (\mu (\nu - 1) - 4)\\
\iff \mu &< \sqrt{\mu (\nu - 1) (\mu (\nu - 1) - 4)}\\
\implies \frac{\nu}{2} &< \frac{\mu (\nu - 1) + \sqrt{\mu (\nu - 1) (\mu (\nu - 1) - 4)}}{2 \mu} = \rho_+^{(\mu, \nu - 1)} ,\\
\frac{2}{\mu} &> \frac{2 (\nu - 1)}{\mu (\nu - 1) + \sqrt{\mu (\nu - 1) (\mu (\nu - 1) - 4)}} = \rho_- ^{(\mu, \nu - 1)}
\end{align*}
so this full region \(\Phi^{(\mu, \nu - 1)} \subseteq \mathfrak D_\infty^{\mu, \nu}\) contains the fundamental region \(\phi_0^{\mu, \nu}\).
\end{proof}

\begin{prop}
\label{prop:step}
If \(\Phi^{3, 2}\) and \(\Phi^{5, 1}\) are full, then so is \(\Phi^{\mu, \nu}\) whenever \(\mu \nu > 4\).
\end{prop}

\begin{proof}
Apply the above repeatedly, using that \(c_{a, b}^{\mu, \nu} = c_{b, a}^{\nu, \mu}\) to reduce the greater of \(\mu\) and \(\nu\) to arrive at one of the base cases.
\end{proof}

\subsection{The base cases}

\begin{lem}
\label{lem:base1}
\(\Phi^{3, 2}\) is full.
\end{lem}

\begin{proof}
Consider a partial deformation of $\mathfrak{D}^{3,2}$ to $\mathfrak{D}^{2,2}$ and $\mathfrak{D}^{1,2}$. In $\mathfrak{D}^{2,2}$ there are rays in directions \((n-1, n )\) with function \((1 + x^{n-1} y^n)^2\) for every \(n\), see Example \ref{expl:scat}. By the change of lattice trick (Proposition \ref{prop:col}, when this hits the horizontal line from the partial deformation, this corresponds to \(\mathfrak D^{n, 2n}\). This diagram has rays in directions \((1, 1), (2,1), \dots, (n,1)\) (see e.g. Example \ref{ex:1k}), which correspond to rays in directions \((n, n), (n+1,n), \dots, (2n-1,n)\). Hence if \(\tfrac{1}{2} < \tfrac{b}{a} < 1\), then a ray in direction \((a, b)\) is present in \(\mathfrak D^{3,2}\). But this contains the fundamental region $\phi_0^{3,2}$. Therefore \(\phi_0^{3, 2}\) is full, and therefore \(\Phi^{3,2}\) is.
\end{proof}

\begin{lem}
\label{lem:base2}
\(\Phi^{5, 1}\) is full.
\end{lem}

\begin{proof}
Consider a partial deformation of \(\mathfrak D^{1, 5}\) to \(\mathfrak D^{1, 3}\) and $\mathfrak{D}^{1,2}$, see Figure \ref{fig:1-5full}. As \(c_{1, 1}^{1, 3} = 3\), we get a subdiagram corresponding to \(\mathfrak D^{3,2}\), with rays \((a, b)\) in \(\mathfrak D^{3,2}\) corresponding to rays \((a, a + b)\) in \(\mathfrak D^{1, 5}\). This maps the slope \(\rho \mapsto \rho + 1\), so sends the dense region \(\Phi^{3,2}\) between
\[ \rho_\pm^{3,2} = \frac{6\pm\sqrt{12}}{6} = 1\pm\frac{1}{\sqrt{3}} \]
to the region between $2-\frac{1}{\sqrt{3}}$ and $2+\frac{1}{\sqrt{3}}$ in \(\mathfrak D^{1, 5}\). This contains the fundamental region $\phi_0^{1,5}$ which is spanned by $2$ and $\tfrac{5}{2}$. So $\phi_0^{1,5}$ is full and hence $\Phi^{1,5}$ is full.
\end{proof}

Now Theorem \ref{thm:full} follows from Lemma \ref{lem:base1}, Lemma \ref{lem:base2} and Proposition \ref{prop:step}.

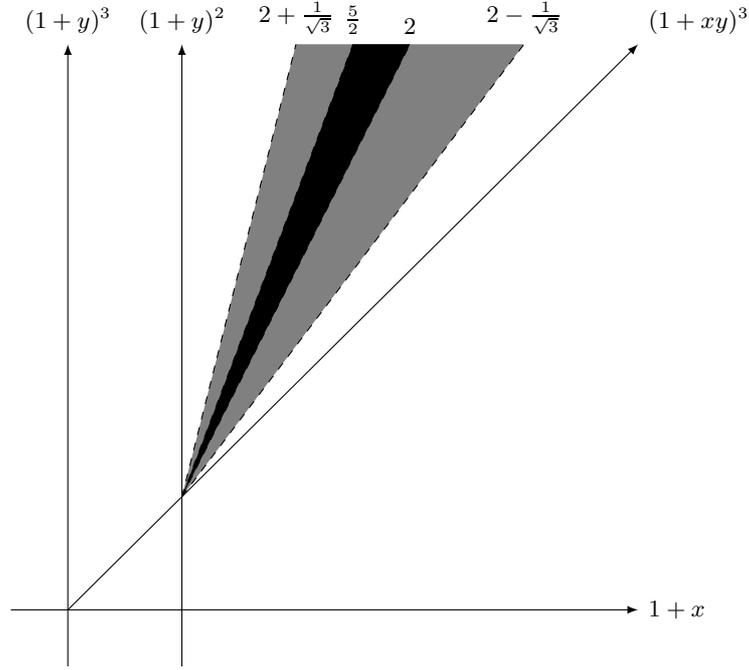
\begin{figure}[h!]
    \centering
    \begin{tikzpicture}[scale=1.5]
        \fill[color=gray] (1, 1) -- (2, 5) -- (4, 5);
        \fill[color=black] (1, 1) -- (2.5, 5) -- (3, 5);
        \draw[->] (-1/2, 0) -- (5, 0) node[right]{\footnotesize \(1 + x\)};
        \draw[->] (0, -1/2) -- (0, 5) node[above]{\footnotesize \((1 + y)^3\)};
        \draw[->] (1, -1/2) -- (1, 5) node[above]{\footnotesize \((1 + y)^2\)};
        \draw[->] (0, 0) -- (5, 5) node[above right]{\footnotesize \((1 + x y)^3\)};
        \draw[dashed] (1, 1) -- (2, 5) node[above]{\footnotesize \(2+\frac{1}{\sqrt{3}}\)};
        \draw[dashed] (1, 1) -- (2.5, 5) node[above]{\footnotesize \(\frac{5}{2}\)};
        \draw[dashed] (1, 1) -- (3, 5) node[above]{\footnotesize \(2\)};
        \draw[dashed] (1, 1) -- (4, 5) node[above]{\footnotesize \(2-\frac{1}{\sqrt{3}}\)};
    \end{tikzpicture}
    \caption{A partial deformation of \(\mathfrak D^{1, 5}\) showing the relevant rays. As any other rays only add nonnegative contributions, they are omitted.}
    \label{fig:1-5full}
\end{figure}

\subsection{Outside the dense region}

\begin{prop}[Theorem \ref{thm:main} (b)]
\label{prop:outside}
Outside of \(\Phi^{\mu, \nu}\), the only rays that occur are those given by mutations of the initial rays. In particular, rays cannot be dense.
\end{prop}
\begin{proof}
Let 
\begin{align*} 
\alpha_0 &= 0, 	& \alpha_1 &= \tfrac{1}{\mu}, 	& \alpha_{n + 1} &= T_2 (\beta_n) \\
\beta_0 &= \infty, 	& \beta_1 &= \nu, 			& \beta_{n + 1} &= T_1 (\alpha_n).
\end{align*}
We know there are no rays with slope \(\alpha_0 < \rho < \alpha_1\) or \(\beta_0 > \rho > \beta_1\), and under mutations, we know if there are no rays with slope \(\alpha_{n - 1} < \rho < \alpha_n\) or \(\beta_{n - 1} > \rho > \beta_n\), there are none with slope \(\alpha_n < \rho < \alpha_{n + 1}\) or \(\beta_n > \rho > \beta_{n + 1}\).

Note also that \(\alpha_0 < \rho_-^{\mu, \nu}\) and \(\beta_0 > \rho_+^{\mu, \nu}\) and that \(T_1, T_2 : \rho_\pm^{\mu, \nu} \mapsto \rho_\mp^{\mu, \nu}\) and are order-reversing, so \(\alpha_n < \rho_-^{\mu, \nu}\) and \(\beta_n > \rho_+^{\mu, \nu}\). So we get bounded monotone sequences \(\alpha_n, \beta_n\), so they converge to \(\alpha,\beta\) respectively. As \(T_2 T_1\) is continuous and maps \(\alpha_n\) to \(\alpha_{n + 2}\) and \(\beta_n\) to \(\beta_{n + 2}\) respectively, \(\alpha, \beta\) must be fixed points of \(T_2 T_1\). But the fixed points of 
\[ T_2 T_1 : \rho \mapsto \nu - \frac{1}{\mu - 1/\rho} \]
are exactly \(\rho_\pm^{\mu, \nu}\). So \(\alpha = \rho_-^{\mu, \nu}\) and \(\beta = \rho_+^{\mu, \nu}\), and we get the claim.
\end{proof}

\section{Recursion}

For a fixed direction vector $(a,b)\in\mathbb{Z}_{>0}^2$ one can get a recursion relation for $c_{a,b}^{\mu,\nu}$ in terms of other coefficients $c_{a',b'}^{\mu',\nu'}$ for which at least one of $a',b',\mu',\nu'$ is smaller than $a,b,\mu,\nu$. This is obtained by considering the partial deformation of $\mathfrak{D}^{\mu,\nu}$ to $\mathfrak{D}^{\mu,\nu-1}$ and $\mathfrak{D}^{\mu,1}$ and describing all scattering that can lead to a ray with direction $(a,b)$.

Here we describe the recursion relation for the easiest families of direction vectors $(1,k)$ and $(2,k)$, for $k\in\mathbb{Z}_{>0}$, and solve it for $(1,k)$. In general the recursion relation will be very difficult and probably won't have a solution in terms of a nice function.

\begin{ex}
\label{ex:1k}
Consider \(c_{1, k}^{\mu, \nu}\). In a partial deformation of \(\mathfrak D^{\mu, \nu}\) (see Figure \ref{fig:pdef1k}), there are only two ways we can get a contribution: directly from a standard subdiagram, or from a ray in direction \((1, k - a)\) striking the vertical line. This is due to the change of lattice method and that all rays occur in the first quadrant. Therefore
\begin{align*}
c_{1, k}^{\mu, \nu} &= c_{1, k}^{\mu, \nu - 1} + c_{1, k}^{\mu, 1} + \sum_{a = 1}^{k - 1} \bigg( c_{1, a}^{c_{1, k - a}^{\mu, \nu - 1}, 1} \bigg)\\
&= c_{1, k}^{\mu, \nu - 1} + c_{1, k}^{\mu, 1} + c_{1, 1}^{c_{1, k - 1}^{\mu, \nu - 1}, 1}\\
&= c_{1, k}^{\mu, \nu - 1} + c_{1, k - 1}^{\mu, \nu - 1} + \begin{cases} \mu &\text{if } k = 1 \\ 0 &\text{otherwise} \end{cases}
\end{align*}
as all rays have slope between \(1/\mu\) and \(\nu\), and \(c_{a, 1}^{\mu, 1} = \binom{\mu}{a}\). We can solve this recurrence relation to obtain \[c_{1, k}^{\mu, \nu} = \mu \binom{\nu}{k} .\]
\end{ex}

\begin{figure}[h!]
\centering
\begin{tikzpicture}[scale=1.2]
\draw[<->] (0.0, -1.0) -- (0.0, 4.0) node[above]{\((1 + y)^{\nu - 1}\)};
\draw[<->] (-1.0, 0.0) -- (4.0, 0.0) node[right]{\((1 + x)^\mu\)};
\draw[<->] (2.0, -1.0) -- (2.0, 4.0) node[above]{\((1 + y)\)};
\draw[->] (0, 0) -- (4, 4) node[above right]{\((1 + x y^k)^{c_{1, k}^{\mu, \nu - 1}}\)};
\draw[->] (0, 0) -- (4, 1) node[right]{\((1 + x y^{k - 1})^{c_{1, k - 1}^{\mu, \nu - 1}}\)};
\draw[->] (2, 0) -- (4, 2) node[right]{\((1 + x y^k)^{c_{1, k}^{\mu, 1}}\)};
\draw[->] (2, 0.5) -- (4, 2.5) node[above right]{\((1 + x y^k)^{c_{1, 1}^{c_{1, k - 1}^{\mu, \nu - 1}, 1}}\)};
\end{tikzpicture}
\caption{A partial deformation of the diagram \(\mathfrak D^{\mu, \nu}\), showing the rays which contribute to \(c_{1, k}^{\mu, \nu}\).}
\label{fig:pdef1k}
\end{figure}
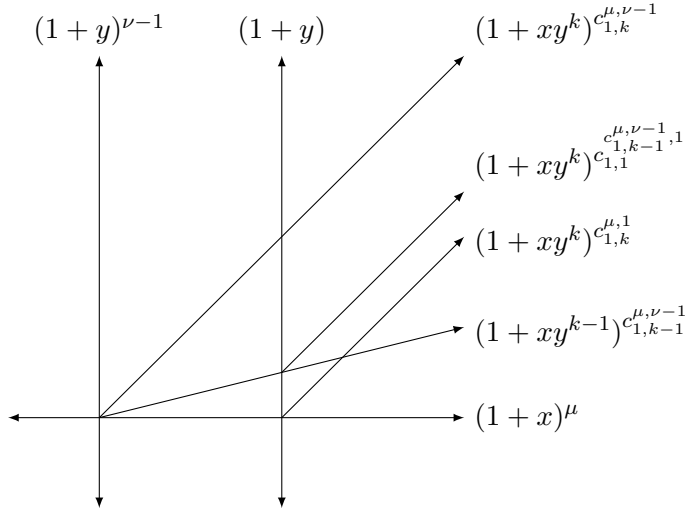

\begin{ex}
\label{ex:2k}
Consider \(c_{2, k}^{\mu, \nu}\). In a partial deformation of \(\mathfrak D^{\mu, \nu}\), there are the following ways we get a contribution:
\begin{enumerate}
    \item Directly from the \(\mathfrak D^{\mu, \nu - 1}\) or \(\mathfrak D^{\mu, 1}\) subdiagrams,
    \item a ray in the direction \((2, a)\) from the \(\mathfrak D^{\mu, \nu - 1}\) subdiagram striking the \((0, 1)\) ray,
    \item a ray in the direction \((1, a)\) from the \(\mathfrak D^{\mu, \nu - 1}\) subdiagram striking the \((0, 1)\) ray,
    \item a ray in the direction \((1, a)\) from the \(\mathfrak D^{\mu, \nu - 1}\) subdiagram striking a ray in the direction \((1, k - a)\) from the \(\mathfrak D^{\mu, 1}\) subdiagram,
    \item from two rays in directions \((1, a), (1, a')\) from the \(\mathfrak D^{\mu, \nu - 1}\) subdiagram, where one ray strikes the \((0, 1)\) ray and produces a new ray which strikes the other ray (see Figure \ref{fig:2kcase}).
\end{enumerate}

\begin{figure}[h!]
\centering
\begin{tikzpicture}[scale=1.2]
\draw[->] (-1, 0) -- (4, 0) node[right]{\((1 + x)^\mu\)};
\draw[->] (0, -1) -- (0, 3) node[above]{\((1 + y)^{\nu - 1}\)};
\draw[->] (2, -1) -- (2, 3) node[above]{\((1 + y)\)};
\draw[->] (0, 0) -- (4, 1) node[right]{\((1 + x^{a'} y)^{\alpha'}\)};
\draw[->] (0, 0) -- (4, 2) node[right]{\((1 + x^a y)^\alpha\)};
\draw[->] (2, 0.5) -- (4, 3) node[above right]{\((1 + x^b y)^\beta\)};
\draw[->] (8/3, 4/3) -- (4, 2.5) node[above right]{\((1 + x^c y^2)^\gamma\)};
\end{tikzpicture}
\caption{Case 5 from Example \ref{ex:2k}.}
\label{fig:2kcase}
\end{figure}
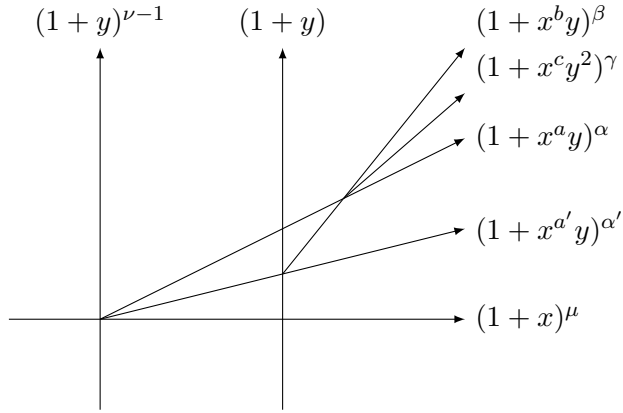
But cases 4 and 5 cannot occur, since for case 4 we would require \(k - a = 1\), and for case 5 we would require \(b = a' + 1 \leq a\). For case 2, the contribution depends on whether \(a\) is even or not:
\begin{itemize}
    \item When \(a\) is even, \((2, a) = (2, 2b)\) is not primitive. The lattice \(M'\) in the change of lattice trick is generated by \((2, 0), (0, 1)\), so the lattice \(N'\) is generated by \((\tfrac{1}{2}, 0), (0, 1)\), and the primitive vector orthogonal to \((2, k)\) is \(n = (\tfrac{k}{2}, -1)\). So \(e(n) = 1\) if \(k\) is even and \(e(n) = 2\) if \(k\) is odd. Similarly \(e(m_2^\ast) = 1\) and \(e(m_1^\ast) = 2\), so we get a contribution of \(c_{1, k - a}^{c_{2, a}^{\mu, \nu - 1}, 2} = c_{2, a}^{\mu, \nu - 1} \binom{2}{k - a}\) if \(k\) is even and \(\frac{1}{2} c_{1, k - a}^{c_{2, a}^{\mu, \nu - 1}, 2} = \frac{1}{2} c_{2, a}^{\mu, \nu - 1} \binom{2}{k - a}\) by Example \ref{ex:1k} if \(k\) is odd. There's no contribution unless \(a = k - 1\) or \(k - 2\).
    \item When \(a\) is odd, we have the above except \(e(m_2^\ast) = 2\), so the contribution is \(\frac{1}{2} c_{1, k - a}^{2 c_{2, a}^{\mu, \nu - 1}, 2} = c_{2, a}^{\mu, \nu - 1} \binom{2}{k - a}\) if \(k\) is odd and \(c_{1, k - a}^{2 c_{2, a}^{\mu, \nu - 1}, 2} = 2 c_{2, a}^{\mu, \nu - 1} \binom{2}{k - a}\) if \(k\) is even, using the change of lattice trick. Again there's no contribution unless \(a = k - 1\) or \(k - 2\).
\end{itemize}
So the total contribution from case 2 is \(c_{2, k - 2}^{\mu, \nu - 1} + 4 c_{2, k - 1}^{\mu, \nu - 1}\) if \(k\) is even and \(c_{2, k - 2}^{\mu, \nu - 1} + c_{2, k - 1}^{\mu, \nu - 1}\) if \(k\) is odd.

For case 3, the lattice \(M' = M\) so we get a contribution of \(c_{2, k - 2a}^{c_{1, a}^{\mu, \nu - 1}, 1}\). But for this to be non-zero, \(k - 2a \leq 2\) by the bound in Proposition \ref{prop:bound}, and furthermore \(k - 2a \neq 2\) as \(f_{1, 1}^{\mu, 1} = (1 + x^2 y)^\mu\) from mutation. So the only contribution comes from when \(k = 2a + 1\) is odd, and it's \(c_{2, 1}^{c_{1, a}^{\mu, \nu - 1}, 1} = \binom{c_{1, a}^{\mu, \nu - 1}}{2} = \binom{\mu \binom{\nu - 1}{a}}{2}\).

So we get the relation \[c_{2, k}^{\mu, \nu} = c_{2, k}^{\mu, \nu - 1} + c_{2, k}^{\mu, 1} + c_{2, k - 2}^{\mu, \nu - 1} + \begin{cases} 4 c_{2, k - 1}^{\mu, \nu} &\text{if } k \equiv 0 \pmod 2 \\ c_{2, k - 1}^{\mu, \nu - 1} + \binom{\mu \binom{\nu - 1}{\lfloor k/2 \rfloor}}{2} &\text{if } k \equiv 1 \pmod 2 \end{cases}\]
\end{ex}
In general, this can be difficult to compute due to the number of ways a ray may occur. Even for \(c_{3, k}^{\mu, \nu}\), we begin to have intersections of rays outside of those involving the starting lines, and we must consider more cases of parity.

\section{Polynomiality of coefficients}

\begin{prop}[Theorem \ref{polythm}]
\label{prop:poly}
The coefficients \(c_{a, b}^{\mu, \nu}\) are polynomial in \(\mu, \nu\), with degrees \(a, b\) respectively. If we expand them in terms of binomial coefficients
\[ c_{a,b}^{\mu,\nu} = \sum_{k=1}^\mu\sum_{l=1}^\nu \lambda_{k,l} \binom{\mu}{k}\binom{\nu}{l}, \]
then $\lambda_{k,l}=0$ for all $k,l$ with $c_{a,b}^{k,l}=0$.
\end{prop}

\begin{proof}
Consider the diagram \[\mathfrak D = \left\{\Bigg(\mathbb R_{\geq 0} (1, 0), \prod_{j = 1}^m (1 + t_{1, j} x) \Bigg), \Bigg(\mathbb R_{\geq 0} (0, 1), \prod_{j = 1}^n (1 + t_{2, j} y) \Bigg) \right\}\] with \(R = \mathbb C \llbracket t_{1, 1}, \dots, t_{1, \mu}, t_{2, 1}, \dots, t_{2, \nu} \rrbracket\). With \(t_{i, j} = 1\), this is the same diagram as $\mathfrak{D}^{\mu,\nu}$ with $t=1$. For ordered partitions \(\mathbf P_1, \mathbf P_2\) of \(\mu, \nu\) respectively, let \(\tilde c_{\mathbf P_1, \mathbf P_2}^{\mu, \nu}\) be the coefficient of \((1 + t_{\mathbf P_1, \mathbf P_2} x^a y^b)\) in \(\mathfrak D_\infty\), where \(t_{i_1 + \dotsb + i_k, j_1 + \dotsb + j_l} = t_{1, i_1} \dotsm t_{1, i_k} t_{2, j_1} \dotsm t_{j_l}\). So \[c_{a, b}^{\mu, \nu} = \sum_{\mathbf P_1, \mathbf P_2} (\tilde c_{\mathbf P_1, \mathbf P_2}^{\mu, \nu}) .\]

However, the length of \(\mathbf P_1\) is at most \(a\), and the length of \(\mathbf P_2\) is at most \(b\). So we can express
\begin{align*}
c_{a, b}^{\mu, \nu} &= \sum_{\ell_1 \leq a} \sum_{\ell_2 \leq b} \sum_{\substack{|\mathbf P_1| = a \\ \ell(\mathbf P_1) = \ell_1}} \sum_{\substack{|\mathbf P_2| = b \\ \ell(\mathbf P_2) = \ell_2}} (\tilde c_{\mathbf P_1, \mathbf P_2}^{\mu, \nu})\\
&= \sum_{\ell_1 \leq a} \sum_{\ell_2 \leq b} \binom{\mu}{\ell_1} \binom{\nu}{\ell_2} \sum_{\substack{|\mathbf P_1| = a \\ \ell(\mathbf P_1) = \ell_1 \\ \mathbf P_1 \in P_{\ell_1}}} \sum_{\substack{|\mathbf P_2| = b \\ \ell(\mathbf P_2) = \ell_2 \\ \mathbf P_2 \in P_{\ell_2}}} ( \tilde c_{\mathbf P_1, \mathbf P_2}^{\ell_1, \ell_2} )
\end{align*}
by symmetry, where \(\mathbf P_i \in P_k\) if the partition contains only the first \(k\) values. But the expression varies with \(\mu, \nu\) only in the binomial expressions. So this is polynomial in \(\mu, \nu\) of degrees at most \(a, b\).

Finally, note that the coefficient of \(\binom{\mu}{a} \binom{\nu}{b}\) is non-zero, as we can totally deform the diagram \(\mathfrak D^{a, b}\), then pick up a contribution from each line, as \(c_{1, 1}^{\mu, \nu} \neq 0\). So \(c_{a, b}^{\mu, \nu}\) is polynomial of degrees \(a, b\).

Now consider $k,l$ with $c_{a,b}^{k,l}=0$. As above we can expand
\[ c_{a,b}^{k,l} = \sum_{\ell_1 \leq k} \sum_{\ell_2 \leq l} \binom{k}{\ell_1} \binom{l}{\ell_2} \sum_{\substack{|\mathbf P_1| = a \\ \ell(\mathbf P_1) = \ell_1 \\ \mathbf P_1 \in P_{\ell_1}}} \sum_{\substack{|\mathbf P_2| = b \\ \ell(\mathbf P_2) = \ell_2 \\ \mathbf P_2 \in P_{\ell_2}}} ( \tilde c_{\mathbf P_1, \mathbf P_2}^{\ell_1, \ell_2} ) \]
Since the coefficients are all non-negative this shows that $\lambda_{\ell_1,\ell_2}=0$ for all $\ell_1\leq k, \ell_2\leq l$.
\end{proof}

For $k,l$ large enough $a,b$ is contained in the dense region for $\mathfrak{D}^{k,l}$ and hence $c_{a,b}^{k,l}\neq 0$ by Theorem \ref{thm:main}. So the condition $c_{a,b}^{k,l}=0$ can only occur for small values of $k,l$. The following lemma makes this precise.

We can also consider unordered partitions $\vec{p}_1,\vec{p}_2$ and expand $c_{a,b}^{\mu,\nu}$ in terms of multinomial coefficients (c.f. \S\ref{S:GW}). Then Proposition \ref{prop:poly} becomes the following.

\begin{cor}
\label{cor:multi}
We have
\[ c_{a,b}^{\mu,\nu} = \sum_{\vec{p}_1,\vec{p}_2} c_{\vec{p}_1,\vec{p}_2} \binom{\mu}{\ell(\vec{p}_1)}\binom{\nu}{\ell(\vec{p}_2)}, \]
where
\[ c_{\vec{p}_1,\vec{p}_2} = \sum_{\mathbf{P}_1,\mathbf{P}_2} c_{\mathbf{P}_1,\mathbf{P}_2}, \]
The sum is over pairs of ordered partitions compatible with $\vec{p}_1,\vec{p}_2$ and $c_{\mathbf{P}_1,\mathbf{P}_2}$ is the exponent of the $1+t^{\mathbf{P}_1,\mathbf{P}_2}$-term in any standard diagram (this doesn't depend on $\mu,\nu$).
\end{cor}

\begin{cor}
\label{cor:van}
If $c_{a,b}^{\ell_1,\ell_2}=0$, then $c_{\vec{p}_1,\vec{p}_2}=0$ for all $\vec{p}_1,\vec{p}_2$ of lengths $\ell_1,\ell_2$.
\end{cor}

\begin{proof}
This follows from Corollary \ref{cor:multi}
\end{proof}

\subsection{Explicit polynomials from interpolation}

It is hard to find the polynomial expressions for $c_{a,b}^{\mu,\nu}$ via deformation and recursion, apart from $c_{1,k}^{\mu,\nu}=\mu\binom{\nu}{k}=\frac{1}{k!}\mu \nu(\nu-1)\cdots (\nu-k+1)$. For example, to get an expression for $c_{2,k}^{\mu,\nu}$, we would have to find a sequence of polynomials solving the recursion in Example \ref{ex:2k}. However, using multivariate polynomial interpolation we can find the polynomial expression for $c_{a,b}^{\mu,\nu}$ from its values for $\mu=0,\ldots,a$ and $\nu=0,\ldots,b$. We show this for $a \leq b \leq 4$ in the following equation. The cases $a>b$ are obtained by symmetry, changing $a\leftrightarrow b$ and $\mu\rightarrow\nu$ simultaneously.

\begin{align*}
c_{1,1}^{\mu,\nu} &= \mu\nu \\
c_{1,2}^{\mu,\nu} &= \tfrac{1}{2}\mu\nu^2 - \tfrac{1}{2}\mu\nu \\
c_{1,3}^{\mu,\nu} &= \tfrac{1}{6}\mu\nu^3 - \tfrac{1}{2}\mu\nu^2 + \tfrac{1}{3}\mu\nu \\
c_{1,4}^{\mu,\nu} &= \tfrac{1}{24}\mu\nu^4 - \tfrac{1}{4}\mu\nu^3 + \tfrac{11}{24}\mu\nu^2 - \tfrac{1}{4}\mu\nu \\
c_{2,2}^{\mu,\nu} &= \mu^2\nu^2 - \mu^2\nu - \mu\nu^2 + \mu\nu \\
c_{2,3}^{\mu,\nu} &= \tfrac{1}{2}\mu^2\nu^3 - \mu^2\nu^2 - \tfrac{1}{3}\mu\nu^3 + \tfrac{1}{2}\mu^2\nu + \tfrac{1}{2}\mu\nu^2 - \tfrac{1}{6}\mu\nu \\
c_{2,4}^{\mu,\nu} &= \tfrac{1}{2}\mu^2\nu^4 - 2\mu^2\nu^3 - \tfrac{1}{3}\mu\nu^4 + \tfrac{5}{2}\mu^2\nu^2 + \mu\nu^3 - \mu^2\nu - \tfrac{2}{3}\mu\nu^2 \\
c_{3,3}^{\mu,\nu} &= \tfrac{5}{2}\mu^3\nu^3 - 7\mu^3\nu^2 - 7\mu^2\nu^3 + \tfrac{9}{2}\mu^3\nu + \tfrac{41}{2}\mu^2\nu^2 + \tfrac{9}{2}\mu\nu^3 - \tfrac{27}{2}\mu^2\nu - \tfrac{27}{2}\mu\nu^2 + 9\mu\nu \\
c_{3,4}^{\mu,\nu} &= \tfrac{2}{3}\mu^3\nu^4 - 2\mu^3\nu^3 - \mu^2\nu^4 + 2\mu^3\nu^2 + \tfrac{5}{2}\mu^2\nu^3 + \tfrac{3}{8}\mu\nu^4 - \tfrac{2}{3}\mu^3\nu - 2\mu^2\nu^2 - \tfrac{3}{4}\mu\nu^3 \\
& + \tfrac{1}{2}\mu^2\nu + \tfrac{11}{24}\mu\nu^2 - \tfrac{1}{12}\mu\nu \\
c_{4,4}^{\mu,\nu} &= \tfrac{13}{6}\mu^4\nu^4 - 4\mu^4\nu^3 - 4\mu^3\nu^4 - \tfrac{3}{2}\mu^4\nu^2 - 12\mu^3\nu^3 - \tfrac{3}{2}\mu^2\nu^4 + \tfrac{10}{3}\mu^4\nu + 60\mu^3\nu^2 + 60\mu^2\nu^3 \\
& + \tfrac{10}{3}\mu\nu^4 - 44\mu^3\nu - \tfrac{1021}{6}\mu^2\nu^2 - 44\mu\nu^3 + \tfrac{335}{3}\mu^2\nu + \tfrac{335}{3}\mu\nu^2 - 71\mu\nu 
\end{align*}

\section{Applications to curves and quivers}

\subsection{Vanishing of Gromov-Witten invariants}

Recall the definition of certain Gromov-Witten invariants from \S\ref{S:GW}. They are related to the coefficients $c_{a,b}^{\mu,\nu}$ of standard scattering diagrams as follows.

\begin{prop}
\label{prop:van}
We have
\[ N_{\vec{p}_1,\vec{p}_2} = \frac{1}{\gcd(\vec{p}_1,\vec{p}_2)}\sum_{l|\gcd(\vec{p}_1,\vec{p}_2)}\frac{(-1)^{l+1}}{l}c_{\frac{1}{l}\vec{p}_1,\frac{1}{l}\vec{p}_2}. \]
In particular, if $\gcd(\vec{p}_1,\vec{p}_2)=1$, then
\[ N_{\vec{p}_1,\vec{p}_2} = c_{\vec{p}_1,\vec{p}_2}. \]
\end{prop}

\begin{proof}
By Definition \ref{def:c} we have
\begin{align*}
\text{log }f_{a,b}^{\mu,\nu} &= \sum_{k>0} c_{ka,kb}^{\mu,\nu}\text{log}(1+x^{ka}y^{kb}) \\
&= \sum_{k>0} c_{ka,kb}^{\mu,\nu} \sum_{l>0} \frac{(-1)^{l+1}}{l} x^{kla}y^{kla} \\
&= \sum_{m>0}\left(\sum_{l|k}\frac{(-1)^{k/l+1}}{k/l}c_{la,lb}^{\mu,\nu}\right)x^{ka}y^{kb}
\end{align*}
Comparing with Corollary \ref{cor:GPS} and Corollary \ref{cor:multi} this gives the desired statement.
\end{proof}

From Proposition \ref{prop:van} and Corollary \ref{cor:van} we get the following.

\begin{cor}
\label{cor:vanishing}
If $c_{a,b}^{\mu,\nu}=0$, then $N_{\vec{p}_1,\vec{p}_2}=0$ for all $\vec{p}_1,\vec{p}_2$ of lengths $\mu,\nu$ with $\gcd(\vec{p}_1,\vec{p}_2)=1$. Intuitively, this means the conditions are too special to be satisfied by any curve.
\end{cor}

\begin{ex}
Consider Gromov-Witten invariants of degree $d$ on $\mathbb{P}^2$, i.e. $(a,b)=(d,d)$. 

We know that for $\mu=1$ or $\nu=1$ we have $f_{(1,1)}^{\mu,\nu}=(1+txy)^\nu$ or $f_{(1,1)}^{\mu,\nu}=(1+txy)^\mu$, respectively. In particular for $d>1$ we have $c_{d,d}^{\mu,\nu}=0$ in this case. By the corollary above this means
\[ N_{d,\vec{p}}=0 \]
whenever $\gcd(d,\vec{p})=1$. 

For $\mu=\nu=2$ we have $f_{(1,1)}^{2,2}=(1-txy)^{-4}$. One easily checks that $c_{d,d}^{2,2}=0$ for $d \neq 2^k$. This implies
\[ N_{a+b,c+d}=0 \]
whenever $\gcd(a,b,c,d)=1$ and $d=a+b=c+d \neq 2^k$.

In the case $d=3$ this gives
\[ N_{3,2+1} = N_{3,1+1+1} = N_{2+1,2+1} = 0, \]
as noticed in \cite{GPS}, \S6.4.
\end{ex}

\subsection{Existence of quiver representations}

Reall the definition of the moduli space of quiver representations $M_{(a,b)}^{(1,0)-st}(Q_{\mu,\nu})$ from \S\ref{S:quiv}. From Theorem \ref{thm:full} and Proposition \ref{prop:chi} we get the following.

\begin{cor}
\label{cor:nonempty}
If $(a,b)$ is in the dense region, then $\chi(M_{(a,b)}^{(1,0)-st}(Q_{\mu,\nu}))>0$ and in particular $M_{(a,b)}^{(1,0)-st}(Q_{\mu,\nu})$ is non-empty.
\end{cor}

\newpage


\end{document}